\documentclass[a4paper, 11pt, twoside, reqno]{amsart}

\usepackage{amsmath,amsfonts,amsthm,amsopn,xcolor,amssymb,enumitem}
\usepackage[colorlinks=true,urlcolor=blue,citecolor=red,linkcolor=blue,linktocpage,pdfpagelabels,bookmarksnumbered,bookmarksopen]{hyperref}
\hypersetup{urlcolor=blue, citecolor=blue, linkcolor=blue}

\usepackage[left=2.61cm,right=2.61cm,top=2.72cm,bottom=2.72cm]{geometry}

\usepackage{color,epsfig,graphics}
\usepackage{latexsym}
\usepackage{graphicx}
\usepackage{mathrsfs}
\usepackage{bm}

\usepackage[T1]{fontenc}

\numberwithin{equation}{section}

\newtheorem{theorem}{Theorem}[section]
\newtheorem{lemma}[theorem]{Lemma}
\newtheorem{remark}[theorem]{Remark}
\newtheorem{proposition}[theorem]{Proposition}

\newtheorem*{claim*}{Claim}

\newcommand{\ef}{\eqref}
\newcommand{\dis}{\displaystyle}

\newcommand{\C}{\mathbb{C}}
\newcommand{\N}{\mathbb{N}}
\newcommand{\R}{\mathbb{R}}

\newcommand{\ep}{\varepsilon}

\newcommand{\intR}{\int_{\mathbb{R}^3}}

\newcommand{\IT}{\int_{\R^3}}
\newcommand{\LT}{L^2(\R^3)}
\newcommand{\HT}{H^1(\R^3, \C)}

\newcommand{\CA}{\mathcal{A}}
\newcommand{\CI}{\mathcal{I}}

\newcommand{\CM}{\mathcal{M}}
\newcommand{\CS}{\mathcal{S}}

\newcommand{\om}{\omega_{\mu}}
\newcommand{\la}{\lambda}

\title
[Schr\"odinger-Poisson system with a doping profile]
{Ground state solutions for Schr\"odinger-Poisson system with a doping profile}

\author[M.~Colin]{Mathieu Colin}
\author[T.~Watanabe]{Tatsuya Watanabe}

\address[M.~Colin]{\newline\indent
University of Bordeaux, CNRS, Bordeaux INP, 
IMB, UMR 5251,  F-33400, Talence, France
\newline\indent
 IMB, UMR 5251,  F-33400, Talence, France
}
\email{mathieu.colin@math.u-bordeaux.fr}

\address[T.~Watanabe]{\newline\indent 
Department of Mathematics, Faculty of Science, Kyoto Sangyo University,
\newline\indent
Motoyama, Kamigamo, Kita-ku, Kyoto-City, 603-8555, Japan}
\email{tatsuw@cc.kyoto-su.ac.jp}

\thanks{}

\date{\today}

\begin{document}

\begin{abstract}
This paper is devoted to the study of 
the nonlinear Schr\"odinger-Poisson system with a doping profile.
We are interested in the existence of ground state solutions
by considering the minimization problem on a Nehari-Pohozaev set.
The presence of a doping profile causes several difficulties, especially 
in the proof of the uniqueness of a maximum point of a fibering map.
A key ingredient is to establish the energy inequality inspired by \cite{TC}.
We also establish the relation between
ground state solutions and $L^2$-constraint minimizers obtained in \cite{CW4}.
When the doping profile is a characteristic function supported on 
a bounded smooth domain, some geometric quantities
related to the domain, such as the mean curvature,
are responsible for the existence of ground state solutions.
\end{abstract}

\maketitle

\section{Introduction}

In this paper, we are concerned with the following nonlinear 
Schr\"odinger-Poisson system:
\begin{equation} \label{eq:1.1}
\begin{cases}
& -\Delta u+\omega u+e \phi u=|u|^{p-1} u \\
& -\Delta \phi =\frac{e}{2} \left( |u|^2- \rho(x) \right)
\end{cases} 
\quad \hbox{in} \ \R^3,
\end{equation}
where $\omega >0 $, $e>0$ and $1<p<5$.
Equation \ef{eq:1.1} appears as a stationary problem for the 
time-dependent nonlinear Schr\"odinger-Poisson system:
\begin{equation} \label{eq:1.2}
\begin{cases}
i \psi_t + \Delta \psi - e\phi \psi + |\psi|^{p-1} \psi =0 
\quad \hbox{in} \ \R_+ \times \R^3, \\
-\Delta \phi = \frac{e}{2}\big( |\psi|^2-\rho(x) \big) 
\quad \hbox{in} \ \R_+ \times \R^3, \\
\psi(0,x)=\psi_0.
\end{cases}
\end{equation}
Indeed when we look for a standing wave of the form: $\psi(t,x)=e^{i \omega t} u(x)$,
we are led to the elliptic problem \ef{eq:1.1}.
In this paper, we are interested in the existence of 
ground state solutions of \ef{eq:1.1}
and their relation with $L^2$-constraint minimizers obtained in \cite{CW4}.

The Schr\"odinger-Poisson system appears 
in various fields of physics, such as quantum mechanics,
black holes in gravitation and plasma physics.
Especially, the Schr\"odinger-Poisson system plays an important role
in the study of semi-conductor theory; see \cite{Je, MRS, Sel},
and then the function $\rho(x)$ is referred as \textit{impurities} 
or a \textit{doping profile}.
The doping profile comes from 
the difference of the number densities of positively charged donor ions
and negatively charged acceptor ions,
and the most typical examples are characteristic functions, step functions 
or Gaussian functions.
Equation \ef{eq:1.1} also appears as a stationary problem
for the Maxwell-Schr\"odinger system.
We refer to \cite{BF, CW, CW2} for the physical background
and the stability result of standing waves for the Maxwell-Schr\"odinger system.
In this case, the constant $e$ describes the strength of the interaction
between a particle and an external electromagnetic field.

The nonlinear Schr\"odinger-Poisson system with $\rho \equiv 0$:
\begin{equation} \label{eq:1.3}
\begin{cases}
& -\Delta u+\omega u+e \phi u=|u|^{p-1} u \\
& -\Delta \phi =\frac{e}{2} |u|^2
\end{cases} 
\quad \hbox{in} \ \R^3
\end{equation}
has been studied widely in the last two decades.
Especially, the existence of non-trivial solutions and 
ground state solutions of \ef{eq:1.3}
has been considered in detail.
Furthermore, the existence of associated $L^2$-constraint minimizers 
depending on $p$ and the size of the mass
and their stability have been investigated as well.
We refer to e.g. \cite{AP, BJL, BS1, CDSS, CW, JL, K2, LZH, R, WL1, WL2, ZZ}
and references therein.
On the other hand, the nonlinear Schr\"odinger-Poisson system with
a doping profile is less studied.
In \cite{DeLeo, DR}, the corresponding 1D problem has been considered.
Moreover, the linear Schr\"odinger-Poisson system 
(that is, the problem \ef{eq:1.1} without $|u|^{p-1}u$)
with a doping profile in $\R^3$ has been studied in \cite{Ben1, Ben2}.
In \cite{CW4}, the authors have investigated the existence of 
stable standing waves for \ef{eq:1.2} by considering the corresponding
$L^2$-minimization problem.
As far as we know, there is no literature 
concerning with the existence of ground state solutions of \ef{eq:1.1},
which is exactly the purpose of this paper.

To state our main results, let us give some notation. 
For $u \in H^1(\R^3,\C)$, 
the energy functional associated with \ef{eq:1.1} is given by
\begin{align}
\mathcal{I}(u) &= \frac{1}{2} \intR |\nabla u |^2 \,dx + \frac{\omega}{2} \intR |u|^2 \,dx
-\frac{1}{p+1} \intR |u|^{p+1} \,d x + e^2 \CA(u). \label{eq:1.4}
\end{align}
Here we denote the nonlocal term by $S(u)=S_0(u)+S_1$ with
\[
\begin{aligned}
S_0(u)(x)&:=(-\Delta)^{-1} \left( \frac{|u(x)|^2}{2} \right) 
=\frac{1}{8 \pi |x|} *|u(x)|^2, \\
S_1(x)&:= (-\Delta)^{-1} \left( \frac{- \rho(x)}{2} \right) 
= -\frac{1}{8 \pi |x|} * \rho(x),
\end{aligned}
\]
and the functional corresponding to the nonlocal term by
\[
\CA(u) 
:= \frac{1}{4} \intR S(u) \big( |u|^2-\rho(x) \big) \,dx
= \frac{1}{32\pi} \intR \intR
\frac{\big( |u(x)|^2-\rho(x) \big) \big( |u(y)|^2-\rho(y) \big)}{|x-y|} \,dx\,dy.
\]
A function $u_0$ is said to be a \textit{ground state solution} (GSS) 
of \ef{eq:1.1}
if $u_0$ has a least energy among all nontrivial solutions of \ef{eq:1.1},
namely $u_0$ satisfies
\[
\mathcal{I}(u_0) = \inf \{ \mathcal{I}(u) \mid u \in \HT
\setminus \{ 0 \}, \ \mathcal{I}'(u)=0 \}.
\]

For the doping profile $\rho$, we assume that
\begin{equation} \label{eq:1.5}
\rho(x) \in L^{\frac{6}{5}}(\R^3) \cap L^q_{loc}(\R^3) \ \text{for some} \ q>3, 
\quad x \cdot \nabla \rho(x) \in L^{\frac{6}{5}}(\R^3),
\quad x \cdot (D^2 \rho(x) x) \in L^{\frac{6}{5}}(\R^3),
\end{equation}
where $D^2 \rho$ is the Hessian matrix of $\rho$, and 
\begin{equation} \label{eq:1.6}
\rho(x) \ge 0, \ \not\equiv 0 \quad \text{for} \ x \in \R^3.
\end{equation}
Typical examples are the Gaussian function $\rho(x) =\ep e^{-\alpha |x|^2}$
and $\rho(x) =\frac{\ep}{1+\alpha |x|^r}$ for $r> \frac{5}{2}$.

In this setting, our first main result can be expressed as follows.

\begin{theorem} \label{thm:1.1}
Suppose that $2<p<5$ and assume \ef{eq:1.5}-\ef{eq:1.6}.
There exists $\rho_0$ independent of $e$, $\rho$ such that if
\[
e^2 \left(\|\rho\|_{L^{\frac{6}{5}}(\R^3)}
+\|x \cdot \nabla \rho\|_{L^{\frac{6}{5}}(\R^3)}
+\| x \cdot (D^2 \rho x ) \|_{L^{\frac{6}{5}}(\R^3)} \right) \le \rho_0,
\]
then \ef{eq:1.1} has a ground state solution $u_0$.
Moreover any ground state solution of \ef{eq:1.1} is real-valued
up to phase shift.
\end{theorem}

When $1<p<2$, we are able to obtain the existence of a \textit{radial}
ground state solution of \ef{eq:1.1}, 
which is a weak result; see Section 6 below.

\smallskip
Our second purpose of this paper is to investigate the relation
between the ground state solution of \ef{eq:1.1} obtained in 
Theorem \ref{thm:1.1} and the $L^2$-constraint minimizer in \cite{CW4}.
There has been significant progress in this relationship in recent years;
see \cite{CW3, DST, JL} for this direction.
Based on the terminology in \cite{CW3, DST},
we call $u_0$ obtained in Theorem \ref{thm:1.1}
an \textit{action ground state solution} of \ef{eq:1.1}.

To state our second main result, we define  
the energy functional $\mathcal{E}: \HT \to \R$ by 
\begin{align*}
\mathcal{E}(u) &=\frac{1}{2} \intR |\nabla u |^2 \,dx 
-\frac{1}{p+1} \intR |u|^{p+1} \,d x + e^2 \mathcal{A}(u). 
\end{align*}
For $\mu>0$, let us consider the minimization problem:
\begin{equation} \label{eq:1.7}
\mathcal{C}(\mu)= \inf_{ u \in B(\mu) } \mathcal{E}(u),
\end{equation}
where $B(\mu) = \{ u \in H^1(\R^3,\C) \mid \| u \|_{L^2(\R^3)}^2=\mu \}$.
In this setting, the constant $\omega$ in \ef{eq:1.1} 
appears as a Lagrange multiplier.
We also define the energy associated with \ef{eq:1.3}:
\[
E_{\infty}(u):=
\frac{1}{2} \| \nabla u \|_2 ^2
-\frac{1}{p+1} \intR |u|^{p+1} \,dx
+\frac{e^2}{4} \intR S_0(u) |u|^2 \,dx.
\]
Indeed if we assume $\rho(x) \to 0$ as $|x| \to \infty$,
\ef{eq:1.3} can be seen as a problem at infinity. 
We define the minimum energy associated with \ef{eq:1.3} by
\begin{equation*} 
c_{e, \infty}(\mu) = c_{\infty}(\mu) := \inf_{u \in B(\mu)} E_{\infty}(u).
\end{equation*}
The existence of minimizers for $c_{e,\infty}(\mu)$ has been studied widely;
we refer to \cite{BJL, BS1, CW, CDSS} and references therein.
Especially in the case $2<p< \frac{7}{3}$, 
$c_{e,\infty}(\mu)$ is attained if $c_{e,\infty}(\mu)<0$.
Moreover there exists $\mu^*=\mu^*(e)>0$ such that
\begin{equation} \label{eq:1.8}
\mu^* = \inf \{ \mu > 0 \mid c_{e, \infty}(\mu) < 0 \}.
\end{equation} 

In \cite{CW4}, under the additional assumptions:
\begin{equation} \label{Ass}
\text{There exist $\alpha>2$ and $C>0$ such that} \ 
\rho(x) \le \frac{C}{1+|x|^{\alpha}} \quad \text{for} \ x \in \R^3,
\end{equation}
\begin{equation} \label{ASS}
\inf_{u \in H^1(\R^3,\C), \| u \|_{\LT}^2=1} 
\IT \left( |\nabla u|^2 + 2e^2 S_1(x) |u|^2 \right) dx = 0,
\end{equation}
it was shown that 
if $2<p<\frac{7}{3}$ and $\mu > 2 \cdot 2^{\frac{1}{2p-4}} \mu^*$,  
there exists $\rho_0= \rho_0(e,\mu)>0$ such that 
if $\|\rho\|_{L^{\frac{6}{5}}(\R^3)}
+\|x \cdot \nabla \rho\|_{L^{\frac{6}{5}}(\R^3)} \le \rho_0$,
the minimization problem \ef{eq:1.7} admits a minimizer $u_{\mu}$
and the associated Lagrange multiplier $\omega=\omega(\mu)$ is positive.

As in \cite{CW3, DST}, we call $u_{\mu}$ an 
\textit{energy ground state solution} of \ef{eq:1.1}.
Our second main result of this paper is the following.

\begin{theorem} \label{thm:1.2}
Let $\mu > 2 \cdot 2^{\frac{1}{2p-4}} \mu^*$ be given
and suppose that $2<p<\frac{7}{3}$.
Under the assumptions in Theorem \ref{thm:1.1}
and \ef{Ass}-\ef{ASS}, the following properties hold.

\noindent
{\rm(i)} 
The energy ground state solution $u_{\mu}$ is 
an action ground state solution of \ef{eq:1.1} with $\omega=\omega_{\mu}$.

\noindent
{\rm(ii)} Let $\Omega(\mu)$ be the set of 
Lagrange multipliers associated with energy ground state solutions for $B_{\mu}$,
namely
\begin{align*}
\Omega(\mu) := \left\{ \right.
\omega_\mu>0 \mid  
&\ \hbox{$\omega_\mu$ is the Lagrange multiplier 
associated with an energy ground state } \\
& \hbox{of \ef{eq:1.1} under the constraint $B_{\mu}$}
\left. \right\},
\end{align*}
and $w_{\mu}$ be an action ground state solution 
of \ef{eq:1.1} with $\omega= \om$.
Then $w_{\mu}$ is an energy ground state solution of \ef{eq:1.1} 
under the constraint $B_{\mu}$.
\end{theorem}

We emphasize that up to authors' knowledge, 
Theorem \ref{thm:1.2} is new even for the case $\rho \equiv 0$.
(See also Remark \ref{rem:5.3} below.)

\smallskip
The assumption \ef{eq:1.5} rules out the case 
$\rho$ is a characteristic function supported on a bounded smooth domain.
Even in this case, we are still able to obtain the existence of 
ground state solutions and their relation
under a smallness condition on some geometric quantities related to the domain;
See Section 7.

\smallskip
Here we briefly explain our strategy and its difficulty.
For the existence of a ground state solution of \ef{eq:1.1},
we adapt a strategy in \cite{AP, R}, that is,
we aim to obtain a ground state solution as a minimizer of 
the \textit{Nehari-Pohozaev} constraint $0=J(u)= 2N(u)-P(u)$,
where $N(u)=0$ is the Nehari identity and $P(u)=0$ is the 
Pohozaev identity for \ef{eq:1.1}.
It is standard to show that for any $u \in \HT$, 
there exists $\la = \la_u>0$ such that $J(u_\la)=0$ provided that
$u_\la (x)=t^2u(\la x)$.
A crucial point is then to show that $\la_u$ is unique.
In \cite{R} for the case $\rho \equiv 0$,
this was carried out by considering up to third derivatives of the fibering map,
which requires an assumption involving third derivative of $\rho$ in our case.
To avoid this difficulty, we follows the idea in \cite{TC} to establish 
the following \textit{energy inequality}:
\begin{align} \label{eq:1.9}
I(u)-I(u_\la)
&\ge \frac{1-\la^3}{3} J(u) +\frac{(1-\la)^2 \omega}{3} \| u\|_2^2
+ \frac{\alpha}{6(p+1)} (1-\la)^2 \| u \|_{p+1}^{p+1} \notag \\
&\quad - \beta (1-\la)^2 e^2 
\left( \| \rho \|_{\frac{6}{5}} + \| x \cdot \nabla \rho \|_{\frac{6}{5}}
+ \| x \cdot (D^2 \rho x) \|_{\frac{6}{5}} \right)
\left( \| u\|_2^2 +\| u \|_{p+1}^2\right)
\end{align}
for any $u \in \HT$, $0 \le \la \le T$ and some $T$, $\alpha$, $\beta>0$.
This key inequality is enough to prove the uniqueness of $\la_u$ for 
$u \in \HT$ with $J(u) \le 0$.
Moreover we can apply the energy inequality to show the relation
between two types of ground state solutions of \ef{eq:1.1}.

As we can easily imagine, if a doping profile $\rho$ is considered, 
scaling arguments do not work straightforwardly 
because of the loss of spatial homogeneity. 
Furthermore the presence of the doping profile $\rho$ 
satisfying \ef{eq:1.5} and \ef{eq:1.6} causes additional difficulties.
Firstly we cannot expect that a remainder term in \ef{eq:1.9}
is zero.
Secondly we are not able to use the concentration function as in \cite{AP}.
See Remark \ref{rem:3.8} and Remark \ref{rem:4.3} for details.

Compared to \cite{CW4}, 
we impose one additional assumption on $\rho$,
which is the $L^{\frac{6}{5}}$-integrability of
$x \cdot (D^2 \rho x)$.
Unlike the existence of minimizers,
we need detailed geometric properties of the functional $\mathcal{I}$.
Especially we apply the second-order Taylor expansion to obtain \ef{eq:1.9}, 
causing that some estimate for $x \cdot (D^2 \rho x)$ is required.

When $\rho$ is a characteristic function, 
further consideration is required
because $\rho$ cannot be weakly differentiable.
In this case, a key of the proof is the \textit{sharp boundary trace inequality}
which was developed in \cite{A},
and a variation of domain
related with the \textit{calculus of moving surfaces} 
due to Hadamard \cite{Gri}.
Then by imposing a smallness condition of some geometric quantities
related to the support of $\rho$, 
we are able to obtain the existence of ground state solutions.

In this paper, we are only concerned with the case $\| \rho \|_{L^{\frac{6}{5}}(\R^3)}$ is small.
For the moment, we don't know what happens if $\| \rho \|_{L^{\frac{6}{5}}(\R^3)}$ is large.
Especially, it is interesting to ask if the non-existence result can be obtained for this case.
It is also worth mentioning that in \cite{Ben1, Ben2} where the problem \ef{eq:1.1} without $|u|^{p-1}u$ was studied, 
the existence of a ground state solution has been shown under the assumption:
\begin{equation} \label{eq:1.11}
\inf \left\{\int_{\mathbb{R}^3}\left(|\nabla \varphi|^2+S_1(x) \varphi^2\right) \,d x, \ 
\int_{\mathbb{R}^3}|\varphi|^2 \,dx =1 \right\}<0
\end{equation}
and for small $\omega >0$.
By Lemma \ref{lem:2.2} below, \ef{eq:1.11} cannot be satisfied when 
$\| \rho \|_{L^{\frac{6}{5}}(\R^3)}$ is small.
It is an another interesting question if we could replace \ef{eq:1.5}-\ef{eq:1.6} by \ef{eq:1.11}
to establish the existence of a ground state solution of \ef{eq:1.1}.

\smallskip
This paper is organized as follows.
In Section 2, we introduce several properties of the energy functional
and some lemmas which will be used later on.
We investigate basic properties of the Nehari-Pohozaev set in Section 3.
In Section 4, we prove the existence of a ground state solution of \ef{eq:1.1}
by applying the concentration compactness principle
and completes the proof of Theorem \ref{thm:1.1}. 
Section 5 is devoted to the relation between 
action ground state solutions and energy ground state solutions. 
We also study the case $1<p<2$ in Section 6.
In Section 7, we finish this paper   
by considering the case $\rho$ is a characteristic function
and present the existence of ground state solutions 
and their relation for this case.

Hereafter in this paper, 
unless otherwise specified, we write $\| u\|_{L^{p}(\R^3)}=\| u\|_p$.
We also set $\| u \|^2 := \| \nabla u \|_2^2 + \| u \|_2^2$.

\section{Variational formulation and preliminaries}

The aim of this section is to prepare several properties of the energy functional
and present intermediate lemmas which will be used later on.

\subsection{Decomposition of the energy} \ 

In this subsection, we rewrite the energy functional $\mathcal{I}$
in a more convenient way.
We put 
\[
A(u)= \| \nabla u \|_2^2, \quad
B(u)= \| u \|_2^2, \quad
C(u)= \| u \|_{p+1}^{p+1},
\]
and decompose $\CI$ in the following way:
\[
\begin{aligned}
\CI(u)
&= \frac{1}{2} A(u) + \frac{\omega}{2} B(u) -\frac{1}{p+1} C(u) \\
&\quad +\frac{e^2}{4} \intR S_0(u) |u|^2 \,dx 
+\frac{e^2}{4} \intR S_1 |u|^2 \,dx 
-\frac{e^2}{4} \intR S_0(u) \rho(x) \,dx 
-\frac{e^2}{4} \intR S_1 \rho(x) \,dx. 
\end{aligned}
\]
Next we define three nonlocal terms:
\begin{align}
D(u)&=\frac{1}{4} \intR S_0(u) |u|^2 \,dx, \label{eq:2.1} \\
E_1(u) &= -\frac{1}{4} \intR S_0(u) \rho(x) \,dx
= -\frac{1}{32\pi} \intR \intR \frac{|u(y)|^2 \rho(x)}{|x-y|} \,dx \,dy
=\frac{1}{4} \intR S_1 |u|^2 \,dx, \label{eq:2.2} \\
F &=-\frac{1}{4} \intR S_1 \rho(x) \,dx. \notag
\end{align}
Note that $F$ is independent of $u \in \HT$.
Then we are able to write $\mathcal{I}$ defined in \ef{eq:1.4} in the following form:
\[
\mathcal{I}(u)= \frac{1}{2} A(u)
+ \frac{\omega}{2} B(u)
-\frac{1}{p+1} C(u) 
+e^2 D(u) + 2e^2 E_1(u) + e^2 F.
\]

Now it is convenient to define $I(u):=\mathcal{I}(u)-e^2 F$, which yields that
\begin{equation} \label{eq:2.3}
I(u) = \frac{1}{2}A(u) + \frac{\omega}{2} B(u) - \frac{1}{p+1}C(u) 
+ e^2 D(u) +2e^2 E_1(u).
\end{equation}
Since $F$ is independent of $u$, 
we have only to consider the 
existence of nontrivial critical point of $I$. 
Recalling that 
\[
S_0(u)(x)=(-\Delta)^{-1} \left(\frac{|u(x)|^2}{2}\right) \geq 0, 
\]
we find that  
\begin{equation} \label{eq:2.4}
A(u), \ B(u), \ C(u), \ 
D(u) \ge 0 \quad \hbox{for all} \ u \in H^1(\R^3,\C).
\end{equation}
For later use, let us also define
\begin{align}
E_2(u) &:= \frac{1}{2} \intR S_0(u) x \cdot \nabla \rho (x) \,dx
= \frac{1}{2} \intR S_2 |u|^2 \,dx, \label{eq:2.5} \\
E_3(u) &:= \frac{1}{2} \intR S_0(u) x \cdot (D^2 \rho (x)x) \,dx
= \frac{1}{2} \intR S_3 |u|^2 \,dx, \notag \\
S_2(x) &= (-\Delta )^{-1} \left( \frac{x \cdot \nabla \rho(x)}{2} \right)
= \frac{1}{8 \pi |x|} * \big( x \cdot \nabla \rho(x) \big), \notag \\
S_3(x) &= (-\Delta )^{-1} \left( \frac{x \cdot (D^2 \rho (x) x )}{2} \right)
= \frac{1}{8 \pi |x|} * \big( x \cdot (D^2 \rho(x) x) \big), \notag 
\end{align}
which is well-defined for $u \in H^1(\R^3,\C)$ by \ef{eq:1.5}.

\subsection{Derivatives of nonlocal terms} \

We next investigate Frechlet derivatives of $D$, $E_1$ and $E_2$
which will be needed later.

\begin{lemma} \label{lem:2.1}

{\rm(i)} \ For $\varphi \in H^1(\mathbb{R}^3, \mathbb{C})$, it holds that
\[
\begin{aligned}
D^{\prime}(u) \varphi &=\int_{\mathbb{R}^3} S_0(u) u \bar{\varphi} \,d x, \\
E_1^{\prime}(u) \varphi 
&= \frac{1}{2} \int_{\mathbb{R}^3} S_1 u \bar{\varphi} \,d x, \\
E_2^{\prime}(u) \varphi &= \int_{\mathbb{R}^3} S_2 u \bar{\varphi} d x.
\end{aligned}
\]
Especially we have
\[
\begin{aligned}
D^{\prime}(u) u &=\int_{\mathbb{R}^3} S_0(u)|u|^2 \,d x =4 D(u), \\
E_1^{\prime}(u) u &=\frac{1}{2} \int_{\mathbb{R}^3} S_1 |u|^2 \,d x
=- \frac{1}{2} \int_{\mathbb{R}^3} S_0(u) \rho(x) \,d x =2 E_1(u), \\
E_2^{\prime}(u) u &=\int_{\mathbb{R}^3} S_2 |u|^2 \,d x
=\int_{\mathbb{R}^3} S_0(u) x \cdot \nabla \rho(x) \,d x =2 E_2(u). 
\end{aligned}
\]

{\rm(ii)} \ It follows that
\[
\begin{aligned}
\lim_{R \to \infty} \int_{B_R(0)} S_0(u) u x \cdot \nabla \bar{u} \,dx 
%D^{\prime}(u) x \cdot \nabla u
&=-\frac{5}{4} \int_{\mathbb{R}^3} S_0(u) |u|^2 \,dx = -5 D(u), \\
\lim_{R \to \infty} \int_{B_R(0)} S_1(u) u x \cdot \nabla \bar{u} \,dx 
%E_1^{\prime}(u) x \cdot \nabla u
&= \frac{5}{2} \int_{\mathbb{R}^3} S_0(u) \rho(x) \,d x
+\frac{1}{2} \intR S_0(u) x \cdot \nabla \rho(x) \,d x \\
&=-10 E_1(u)+ E_2(u), \\
\lim_{R \to \infty} \int_{B_R(0)} S_2(u) u x \cdot \nabla \bar{u} \,dx 
%E_2^{\prime}(u) x \cdot \nabla u
&= -3 \int_{\mathbb{R}^3} S_0(u) x \cdot \nabla \rho(x) \,d x
-\frac{1}{2} \intR S_0(u) x \cdot (D^2 \rho(x)x) \,d x \\
&=-6 E_2(u)-E_3(u).
\end{aligned}
\]
\end{lemma}

\begin{proof}
(i) We observe from \ef{eq:2.1} that
\[
\begin{aligned}
D^{\prime}(u) \varphi
&= \frac{1}{4} \int_{\mathbb{R}^3} S_0'(u) \varphi |u|^2 \,dx
+\frac{1}{2} \int_{\R^3} S_0(u) u \bar{\varphi} \,dx \\
&= \frac{1}{16\pi} \intR \intR 
\frac{u(y) \overline{\varphi(y)} |u(x)|^2}{|x-y|} \,dy \,dx 
+\frac{1}{2} \int_{\R^3} S_0(u) u \bar{\varphi} \,dx \\
&= \frac{1}{2} \int_{\mathbb{R}^3} S_0(u) u \bar{\varphi} \,d x
+\frac{1}{2} \int_{\mathbb{R}^3} S_0(u) u \bar{\varphi} \,d x
=\int_{\mathbb{R}^3} S_0(u) u \bar{\varphi} \,d x.
\end{aligned}
\]
The derivatives of $E_1$ and $E_2$ can be derived readily by \ef{eq:2.2} and \ef{eq:2.5}.

(ii) 
%suppose that $x \cdot \nabla u \in H^1(\R^3,\C)$.
By the divergence theorem and the fact $S_0(u)|u|^2 \in L^1(\R^3)$,
arguing as in \cite{BL, Ca}, one has
\[
\begin{aligned}
&\lim_{R \to \infty} \int_{B_R(0)} S_0(u) u x \cdot \nabla \bar{u} \,dx \\
%D^{\prime}(u) x \cdot \nabla u 
&= \lim_{R \to \infty} \left\{ \frac{1}{2} 
\int_{\partial B_R(0)} S_0(u)|u|^2 x \cdot n \,dS
-\frac{3}{2} \int_{B_R(0)} S_0(u)|u|^2 \,d x
-\frac{1}{2} \int_{B_R(0)} |u|^2 x \cdot \nabla S_0(u) \,d x \right\} \\
%= \int_{\R^3} S_0(u) u x \cdot \nabla \bar{u} \,d x 
&= -\frac{3}{2} \int_{\R^3} S_0(u)|u|^2 \,d x
-\frac{1}{2} \intR |u|^2 x \cdot \nabla S_0(u) \,d x .
\end{aligned}
\]
Using 
\[
-\Delta S_0(u)=\frac{1}{2} |u|^2 \quad \text{and} \quad 
\intR | \nabla S_0(u)|^2 \,dx = \frac{1}{2} \intR S_0(u) |u|^2 \,dx,
\]
we find that
\[
\begin{aligned}
\frac{1}{2} \intR |u|^2 x \cdot \nabla S_0(u) \,d x 
& =\int_{\mathbb{R}^3} \nabla S_0(u) \cdot 
\nabla \big( x \cdot \nabla S_0(u) \big) d x \\
&= \int_{\mathbb{R}^3} |\nabla S_0(u)|^2 \,d x
+\int_{\mathbb{R}^3} \nabla S_0(u) \cdot \left(D^2 S_0(u) x\right) dx \\
&= \int_{\mathbb{R}^3} |\nabla S_0(u)|^2 \,d x
-\intR S_0(u) \Delta S_0(u) \,d x
-\intR S_0(u) x \cdot \nabla \big(\Delta S_0(u) \big) dx \\
&= \intR |\nabla S_0(u)|^2 \,dx 
+\frac{1}{2} \int_{\mathbb{R}^3} S_0(u)|u|^2 \,dx
+\frac{1}{2} \intR S_0(u) x \cdot \nabla |u|^2 \,dx \\
&= -\frac{1}{2} \intR S_0(u) |u|^2 \,dx
-\frac{1}{2} \int_{\R^3} |u|^2 x \cdot \nabla S_0(u) \,dx.
\end{aligned}
\]
This implies that 
\[
\intR |u|^2 x \cdot \nabla S_0(u) \,d x
= - \frac{1}{2} \intR S_0(u) |u|^2 \,dx
\]
and hence
\[
\lim_{R \to \infty} \int_{B_R(0)} S_0(u) u x \cdot \nabla \bar{u} \,dx 
%D^{\prime}(u) x \cdot \nabla u
=-\frac{5}{4} \int_{\R^3} S_0(u) |u|^2 \,dx.
\]
Next since $\Delta S_1= \frac{\rho}{2}$, it holds that
\[
\begin{aligned}
\frac{1}{2} \int_{\mathbb{R}^3} |u|^2 x \cdot \nabla S_1 \,dx
&= \int_{\mathbb{R}^3} \nabla S_0(u) \cdot \nabla (x \cdot \nabla S_1 ) \,d x \\
&= -\int_{\mathbb{R}^3} S_0(u) \rho(x) \,d x
-\frac{1}{2} \int_{\mathbb{R}^3} S_0(u) x \cdot \nabla \rho(x) \,d x,
\end{aligned}
\]
yielding that
\[
\begin{aligned}
\lim_{R \to \infty} \int_{B_R(0)} S_1(u) u x \cdot \nabla \bar{u} \,dx 
%E_1^{\prime}(u) x \cdot \nabla u 
&= -\frac{3}{2} \int_{\mathbb{R}^3} S_1 |u|^2 \,d x
-\frac{1}{2} \int_{\mathbb{R}^3} |u|^2 x \cdot \nabla S_1 \,d x \\
&= \frac{5}{2} \int_{\mathbb{R}^3} S_0(u) \rho(x) \,d x
+\frac{1}{2} \int_{\mathbb{R}^3} S_0(u) x \cdot \nabla \rho(x) \,d x .
\end{aligned}
\]
Similarly from $\Delta S_2= - \frac{x \cdot \nabla \rho}{2}$, 
we obtain
\[
\frac{1}{2} \int_{\R^3} |u|^2 x \cdot \nabla S_2 \,dx
= \frac{3}{2} \int_{\R^3} S_0(u) x \cdot \nabla \rho(x) \,dx 
+\frac{1}{2} \int_{\R^3} S_0(u) x \cdot (D^2 \rho(x) x ) \,dx,
\]
and hence
\[
\begin{aligned}
\lim_{R \to \infty} \int_{B_R(0)} S_2(u) u x \cdot \nabla \bar{u} \,dx 
%E_2^{\prime}(u) x \cdot \nabla u 
& =\int_{\mathbb{R}^3} S_2 u x \cdot \nabla \bar{u} \,dx 
=-\frac{3}{2} \int_{\mathbb{R}^3} S_2 |u|^2 \,d x
-\frac{1}{2} \int_{\mathbb{R}^3} |u|^2 x \cdot \nabla S_2 \,dx \\
&= -3 \int_{\mathbb{R}^3} S_0(u) x \cdot \nabla \rho(x) \,d x
-\frac{1}{2} \int_{\mathbb{R}^3} S_0(u) x \cdot (D^2\rho(x) x ) \,dx.
\end{aligned}
\]
This completes the proof.
\end{proof}

\subsection{Estimates of nonlocal terms} \ 

This subsection is devoted to present
estimates for the nonlocal terms.

\begin{lemma} \label{lem:2.2}

For any $u \in H^1(\R^3,\C)$, 
$S_0$, $D$, $E_1$, $E_2$ and $E_3$ satisfy the estimates: 
\[
\begin{aligned}
\| S_0(u) \|_{6} 
&\le C\| \nabla S_0(u) \|_{2} 
\le C \| u\|_{{\frac{12}{5}}}^2
\le C \| u \|^2, \\
\| \nabla S_0(u) \|_2
&\le C \| u \|_2^{\frac{5p-7}{3(p-1)}} \| u \|_{p+1}^{\frac{p+1}{3(p-1)}}
\le C( \| u \|_2^2 + \| u \|_{p+1}^2)
\quad \text{if} \ \ 2<p<5, \\
D(u) &\leq C \| S_0(u) \|_6 \| u \|_{\frac{12}{5}}^2
\le C \| u \|^4, \\
|E_1(u)| &\le \frac{1}{4} \|S_0(u) \|_{6} \| \rho \|_{\frac{6}{5}} 
\le C \| \rho \|_{\frac{6}{5}} \| u \|^2, \\
|E_2(u)| &\le \frac{1}{2} \| S_0(u) \|_{6} 
\| x \cdot \nabla \rho \|_{\frac{6}{5}}
\le C \| x \cdot \nabla \rho \|_{\frac{6}{5}} \| u \|^2, \\
|E_3(u)| &\le \frac{1}{2} \| S_0(u) \|_{6} 
\| x \cdot (D^2 \rho x) \|_{\frac{6}{5}}
\le C \| x \cdot (D^2 \rho x) \|_{\frac{6}{5}} \| u \|^2.
\end{aligned}
\]

\end{lemma}
For the proof of the inequality on $S_0(u)$, we refer to \cite{R}.
The other estimates can be obtained by the H\"older inequality
and the Sobolev inequality.

\subsection{Convergence properties of nonlocal terms} \ 

Next we introduce the Brezis-Lieb type convergence result for $D(u)$.

\begin{lemma} \label{lem:2.3}
Assume that $u_n \rightharpoonup u$ in $H^1(\R^3, \C)$. Then it follows that
\begin{align*}
\lim_{n \to \infty} \left\{ D(u_n-u) -D(u_n)+D(u) \right\} =0.
\end{align*}
Moreover if $u_n \to u$ in $L^{\frac{12}{5}}(\R^3)$, we also have
\[
S_0(u_n) \to S_0(u) \ \ \text{in} \ L^6(\R^3)
\quad \text{and} \quad 
\lim_{n \to \infty} D(u_n) = D(u).
\]
\end{lemma}

\begin{proof}
The proof can be found in \cite[Lemma 2.2]{ZZ}.
\end{proof}

As for $E_1$ and $E_2$, we have the following compactness property,
which follows from the integrability of $\rho$ and $x \cdot \nabla \rho$.

\begin{lemma} \label{lem:2.4}
Assume that $u_n \rightharpoonup u$ in $\HT$.
Then it follows that
\[
\lim_{n \to \infty} E_1(u_n) = E_1(u)
\quad \text{and} \quad
\lim_{n \to \infty} E_2(u_n)=E_2(u).
\]
\end{lemma}

\begin{proof}
First we observe that since $u_n$ converges weakly in $\HT$,
there exists $C>0$ such that $\| u_n \| \le C$.
Moreover passing to a subsequence, we may assume that
$u_n \to u$ in $L^r_{loc}(\R^3)$ for $2 \le r <6$.
Then similarly as Lemma \ref{lem:2.3}, 
one can see that $S_0(u_n) \to S_0(u)$ in $L^6_{loc}(\R^3)$.
Furthermore since $\rho \in L^{\frac{6}{5}}(\R^3)$, 
for any $\ep >0$, there exists $R_{\ep}>0$ such that
\[
\left( \int_{|x| \ge R_{\ep}} | \rho(x)|^{\frac{6}{5}} \,dx \right)^{\frac{5}{6}} < \ep.
\]

Now from \ef{eq:2.2} and the H\"older inequality, it follows that
\begin{align*}
| E_1(u_n) - E_1(u) |
&\le \frac{1}{4} \intR \left| S_0(u_n) -S_0(u) \right| |\rho| \,dx \\
&= \frac{1}{4} \int_{|x| \le R_\ep} \left| S_0(u_n) -S_0(u) \right| |\rho| \,dx
+ \frac{1}{4} \int_{|x| \ge R_\ep} \left| S_0(u_n) -S_0(u) \right| |\rho| \,dx \\
&\le \frac{1}{4} \left( \int_{|x| \le R_\ep} 
|S_0(u_n)-S_0(u)|^6 \,dx \right)^{\frac{1}{6}} \| \rho \|_{\frac{6}{5}} \\
&\quad +\frac{1}{4} \left( \| S_0(u_n) \|_6 + \| S_0(u) \|_6 \right)
\left( \int_{|x| \ge R_{\ep}} | \rho(x)|^{\frac{6}{5}} \,dx \right)^{\frac{5}{6}}.
\end{align*}
Thus by Lemma \ref{lem:2.2} and Lemma \ref{lem:2.3}, we obtain
\[
\limsup_{n \to \infty} | E_1(u_n) - E_1(u) |
\le \frac{1}{4} \left( C + \| S_0(u) \|_6 \right) \ep.
\]
Since $\ep$ is arbitrary, we deduce that
$E_1(u_n) \to E_1(u)$ as $n \to \infty$.
$E_2(u)$ can be treated in a similar manner.
\end{proof}

\subsection{Scaling properties} \

In this subsection, we collect scaling properties of the nonlocal terms $D$ and $E_1$. 
For $a$, $b \in \R$ and $\la>0$, 
let us adapt the scaling $u_{\la} (x) := \la^a u( \la^b x)$. 
Then we have
\[
\begin{aligned}
S_0(u_{\la})(x)
&=\frac{1}{8 \pi } \intR \frac{|u_{\la}(y)|^2}{\left| x-y\right|} \,d y 
=\frac{\la^{2a}}{8 \pi } 
\intR \frac{|u(\la^b y)|^2}{\left| x- y\right|} \,dy  
= \frac{\la^{2a+b}}{8 \pi} \intR 
\frac{|u(\la^b y)|^2}{| \la^b x- \la^b y|} \,dy \\
&\stackrel{y=\la^{-b} z}{=} \frac{\la^{2a-2b}}{8\pi} \intR 
\frac{|u(z)|^2}{|\la^b x -z|} \,dz. 
\end{aligned}
\]
Thus one finds that
\begin{align}
S_0(u_{\la})(x) &= \la^{2 a -2 b} S_0(u) ( \la^{b} x), \nonumber \\
D(u_\la) &= \la^{4 a-5 b} D(u),  \label{eq:2.6} \\
E_1(u_\la) 
& =-\frac{1}{4} \intR S_0(u_{\la}) (x) \rho(x) \,d x 
=- \frac{\la^{2 a-2 b}}{4} \intR S_0(u)(\la^{b} x) \rho( x) \,dx \nonumber \\
&=- \frac{\la^{2 a-5 b}}{4} \intR S_0(u)(x) \rho (\la^{-b} x) \,d x. \label{eq:2.7}
\end{align}
By the H\"older inequality, it follows that
\begin{equation} \label{eq:2.8}
|E_1(u_{\la})| \le \frac{\la^{2a-5b}}{4} 
\| S_0(u)\|_6 \| \rho( \la^{-b} \cdot ) \|_{\frac{6}{5}}
\le C \la^{2a-\frac{5}{2}b} \| \rho \|_{\frac{6}{5}} \| u\|^2.
\end{equation}

\subsection{Nehari and Pohozaev identities} \ 

This subsection is devoted to establish the
Nehari identity and the Pohozaev identity associated with \ef{eq:1.1}.

\begin{lemma} \label{lem:2.5}
Let $u \in H^1(\R^3,\C)$ be a weak solution of \ef{eq:1.1}.
Then $u$ satisfies the Nehari identity $N(u)=0$ and 
the Pohozaev identity $P(u)=0$, where
\begin{align}
N(u)&= A(u) + \omega B(u) - C(u) + 4e^2 D(u) + 4e^2 E_1(u), \label{eq:2.9} \\
P(u)&= \frac{1}{2} A(u) + \frac{3\omega}{2} B(u) - \frac{3}{p+1} C(u)
+5e^2 D(u) + 10e^2 E_1(u) - e^2 E_2(u). \label{eq:2.10}
\end{align}

\end{lemma}

\begin{proof}
First by Lemma \ref{lem:2.1} (i), one has
\begin{align*}
0= I'(u)u 
&= \| \nabla u \|_2^2 + \omega \| u \|_2^2 - \| u \|_{p+1}^{p+1}
+e^2 D'(u)u + 2e^2 E_1'(u)u \\
&= A(u) + \omega B(u) - C(u) + 4e^2 D(u) + 4e^2 E_1(u).
\end{align*}
Next by Lemma \ref{lem:2.1} (ii), formally it holds that
\begin{align*}
0= I'(u) x \cdot \nabla u
&= - \frac{1}{2} \| \nabla u\|_2^2 - \frac{3\omega}{2} \| u \|_2^2
+\frac{3}{p+1} \| u \|_{p+1}^{p+1} 
+e^2 D'(u) x \cdot \nabla u + 2e^2 E_1'(u)x \cdot \nabla u \\
&= -\frac{1}{2}A(u) - \frac{3\omega}{2} B(u) + \frac{3}{p+1} C(u)
-5e^2 D(u) -10 e^2 E_1(u) + e^2 E_2(u).
\end{align*}

A rigorous proof can be done by establishing the $C^{1,\alpha}_{loc}$-regularity of 
any weak solution of \ef{eq:1.1} for some $\alpha \in (0,1)$. 
Note that since $\rho \in L^{q}_{loc}(\R^3)$ for some $q>3$, 
it follows by the elliptic regularity theory that $S_1 \in W^{2,q}_{loc}(\R^3)
\hookrightarrow C^{1,\alpha}_{loc}(\R^3)$. 
The smoothness of $u$ can be shown similarly by applying the elliptic regularity theory.
Then multiplying $x \cdot \nabla \bar{u}$ and $e x \cdot \nabla S(u)$ by \ef{eq:1.1} 
respectively, 
integrating them over $B_R(0)$ and passing to a limit $R \to \infty$,
we are able to prove \ef{eq:2.10} as in \cite{BL, Ca}.
\end{proof}

\section{Properties of Nehari-Pohozaev set}

In this section, we investigate fundamental properties 
of a Nehari-Pohozaev set, which we will use to obtain the existence 
of a ground state solution of \ef{eq:1.1}.

Now let us define
\[
J(u):= 2N(u) - P(u).
\]
From \ef{eq:2.9} and \ef{eq:2.10}, it holds that
\begin{equation} \label{eq:3.1}
J(u) =\frac{3}{2} A(u)+\frac{\omega}{2} B(u)
-\frac{2p-1}{p+1} C(u) +3 e^2 D(u)-2 e^2 E_1(u)+e^2 E_2(u).
\end{equation}
We also denote the Nehari-Pohozaev set by $\CM$:
\[
\CM =\left\{ u \in \HT \setminus \{ 0 \} \mid J(u)=0 \right\} .
\]
By Lemma \ref{lem:2.5}, one knows that 
any weak solution of \ef{eq:1.1} belongs to $\CM$.
We will show later that a minimizer $I|_{\CM}$ is 
actually a ground state solution of \ef{eq:1.1}.

For this purpose, we begin with the following lemma.
Hereafter we let
\[
u_\la(x) := \la^2 u(\la x) \quad \text{for} \ u \in \HT \setminus \{0\}
\ \text{and} \ \la>0.
\]

\begin{lemma} \label{lem:3.1}
Suppose that $2<p<5$.
For any $u \in \HT \setminus \{ 0 \}$, 
there exists $\la_u>0$ such that $u_{\la_u} \in \CM$.
Especially the set $\CM$ is non-empty.
\end{lemma}

\begin{proof}
Taking $a=2$ and $b=1$, we have from \ef{eq:2.3}, \ef{eq:2.6} and \ef{eq:2.7} that
\begin{equation} \label{eq:3.2}
f(\la):= I(u_\la)
=\frac{\la^3}{2} A(u) +\frac{\omega \la}{2} B(u)
-\frac{\la^{2p-1}}{p+1} C(u) +e^2 \la^3 D(u)
-\frac{e^2 \la^{-1}}{2} \int_{\mathbb{R}^3} S_0(u) \rho (\la^{-1} x) \,dx.
\end{equation}
By \ef{eq:2.4} and \ef{eq:2.8}, it follows that
\[
f(\la) \ge \frac{\omega \la}{2} B(u) -\frac{\la^{2p-1}}{p+1} C(u)
-C e^2 \la^{\frac{3}{2}} \| \rho \|_{\frac{6}{5}} \| u\|^2,
\]
from which we deduce that $f(\la)>0$ for small $\la>0$.
On the other hand since $2p-1 >3$, one also finds that
$f(\la) \to - \infty$ as $\la \to \infty$.
This implies that there exists a maximum point $\la=\la_u>0$ so that $f'(\la_u)=0$.

By a direct computation, $0= \la_u f'(\la_u)$ shows that
\[
\begin{aligned}
0 &= \frac{3\la_u^3}{2} A(u) +\frac{\omega \la_u}{2} B(u)
-\frac{2 p-1}{p+1} \la_u^{2 p-1} C(u) + 3e^2 \la_u^3 D(u) \\
&\quad +\frac{e^2 \la_u^{-1}}{2} \int_{\mathbb{R}^3} S_0(u) \rho(\la_u^{-1} x) \,d x
+\frac{e^2 \la_u^{-1}}{2} \int_{\mathbb{R}^3} S_0(u) 
\la_u^{-1}x \cdot \nabla \rho (\la_u^{-1} x) \,d x.
\end{aligned}
\]
Using \ef{eq:2.6} and \ef{eq:2.7} again, 
we find that $0= J(u_{\la_u})$, which ends the proof.
\end{proof}

\begin{lemma} \label{lem:3.2}
Suppose that $2<p<5$.
There exist $\rho_0$, $\delta_0$ and $\alpha_0>0$
independent of $e$, $\rho$ such that if
\[
e^2 \left( \| \rho \|_{\frac{6}{5}}
+\| x \cdot \nabla \rho \|_{\frac{6}{5}} \right) \le \rho_0,
\]
then it holds that
\[
J(u) \ge \alpha_0 \| u \|^2 \quad \text{for any} \ \
u \in \HT \ \ \text{with} \ \ 0 < \| u \| < \delta_0.
\]

Furthermore if $J(u) \le 0$, 
there exists $\delta_1>0$ independent of $e$, $\rho$ such that 
\begin{equation} \label{eq:3.3}
\| u \|_{p+1} \ge \delta_1.
\end{equation}

\end{lemma}

\begin{proof}
By the Sobolev inequality, Lemma \ref{lem:2.2} and from \ef{eq:3.1}, one has
\[
J(u) \ge \frac{\min \{3, \omega\}}{2} \| u \|^2  - C_1 \| u \|^{p+1}
-C_2 e^2 \left( \| \rho \|_{\frac{6}{5}} 
+ \| x \cdot \nabla \rho\|_{\frac{6}{5}} \right) \| u\|^2,
\]
where $C_1$, $C_2>0$ are independent of $e$ and $\rho$.
Thus if
\[
e^2 \left( \| \rho \|_{\frac{6}{5}} 
+ \| x \cdot \nabla \rho\|_{\frac{6}{5}} \right)
\le \frac{1}{4 C_2} \min \{3, \omega \},
\]
it follows that 
\[
J(u) \ge \frac{ \min \{ 3, \omega \} }{4} \| u\|^2- C_1\| u\|^{p+1}.
\]
Putting $\delta_0= \left( \frac{\min \{3, \omega \}}{8 C_1}\right)^{\frac{1}{p-1}}$,
we obtain
\[
J(u) \ge \frac{ \min \{3, \omega \}}{8} \| u\|^2 
\quad \text{for} \ 0< \| u \| < \delta_0.
\]

Next suppose that $J(u) \le 0$. 
Then by Lemma \ref{lem:2.2}, one finds that
\[
\begin{aligned}
0 \ge J(u) &\ge \frac{3}{2} \|\nabla u \|_2^2
+\frac{\omega}{2}\| u \|_2^2 -\frac{2p-1}{p+1}\| u\|_{p+1}^{p+1} 
-C_3 e^2 \left( \|\rho\|_{\frac{6}{5}} +\|x \cdot \nabla \rho\|_{\frac{6}{5}}\right)
\left( \| u\|_2^2+ \| u \|_{p+1}^2\right) \\
&\ge \frac{\min \{3, \omega \}}{2} \| u\|^2 -\frac{2 p-1}{p+1}\| u\|_{p+1}^{p+1} 
-C_3 e^2 \left( \| \rho \|_{\frac{6}{5}}+\|x \cdot \nabla \rho \|_{\frac{6}{5}}\right)
\left(\| u \|^2+\| u\|_{p+1}^2\right).
\end{aligned}
\]
Choosing 
\[
e^2 \left(\| \rho\|_{\frac{6}{5}}+\|x \cdot \nabla \rho \|_{\frac{6}{5}}\right) 
\leq \frac{1}{4 C_3} \min \{3, \omega \},
\]
and using the Sobolev inequality, we obtain
\begin{align*}
0 \ge J(u) 
&\ge \frac{\min \{ 3, \omega \}}{4} \|u \|^2
- \frac{2p-1}{p+1} \| u \|_{p+1}^{p+1} 
-C_3 e^2 \left( \| \rho \|_{\frac{6}{5}}+\|x \cdot \nabla \rho \|_{\frac{6}{5}}\right)
\| u \|_{p+1}^2 \\
&\ge \left\{ \frac{C_4 \min \{3, \omega \}}{4} 
-C_3 e^2 \left( \| \rho \|_{\frac{6}{5}}+\|x \cdot \nabla \rho \|_{\frac{6}{5}}\right) 
\right\} \| u \|_{p+1}^2 - \frac{2p-1}{p+1} \| u\|_{p+1}^{p+1}.
\end{align*}
Thus if 
\[
e^2 \left( \| \rho \|_{\frac{6}{5}}+\|x \cdot \nabla \rho \|_{\frac{6}{5}}\right) 
\le \frac{C_4}{ 8 C_3} \min \{ 3, \omega \},
\]
it holds that
\[
0 \ge J(u) \ge 
\left( \frac{ C_4 \min \{3, \omega \}}{8} 
-\frac{2p-1}{p+1}\| u\|_{p+1}^{p-1} \right) \| u\|_{p+1}^2,
\]
which implies that
\[
\| u \|_{p+1} \ge \left( \frac{(p+1) C_4 \min \{ 3, \omega \}}{8(2p-1)} 
\right)^{\frac{1}{p-1}}.
\]
This completes the proof.
\end{proof}

Now by Lemma \ref{lem:3.2}, we can define
\begin{equation} \label{eq:3.4}
\sigma := \inf_{u \in \CM} I(u).
\end{equation}

\begin{lemma} \label{lem:3.3}
Suppose that $2<p<5$. 
There exist $\rho_0$,
$\alpha_1>0$ independent of $e$, $\rho$ such that if
\[
e^2\left(\| \rho \|_{\frac{6}{5}}+\| x \cdot \nabla \rho \|_{\frac{6}{5}}\right) 
\le \rho_0,
\]
then it holds that 
\[
I(u) \ge \alpha_1 \| u \|^2 \quad \text{for any} \ u \in \CM.
\]
Especially $\sigma$ is positive.
\end{lemma}

\begin{proof}
From \ef{eq:2.3} and \ef{eq:3.1}, one has
\[
\begin{aligned}
I(u) &= \frac{1}{2} A(u)+\frac{\omega}{2} B(u) -\frac{1}{p+1} C(u)
+e^2 D(u)+2 e^2 E_1(u), \\
0 &= \frac{3}{2} A(u)+\frac{\omega}{2} B(u) -\frac{2p-1}{p+1} C(u)
+3 e^2 D(u)-2 e^2 E_1(u) +e^2 E_2(u),
\end{aligned}
\]
from which we deduce that 
\[
(2 p-1) I(u) = (p-2) A(u) +(p-1) \omega B(u) +2(p-2) e^2 D(u) 
+4 p e^2 E_1(u)-e^2 E_2(u).
\]
By Lemma \ref{lem:2.2} and from \ef{eq:2.4}, it follows that
\[
(2p-1)I(u) 
\ge \min \{ p-2, (p-1)\omega \} \| u \|^2
-C_1 e^2 \left(\| \rho \|_{\frac{6}{5}}+\| x \cdot \nabla \rho \|_{\frac{6}{5}}\right)
\| u \|^2,
\]
where $C_1>0$ is independent of $e$ and $\rho$.
Thus if 
\[
e^2\left(\| \rho \|_{\frac{6}{5}}+\| x \cdot  \nabla \rho \|_{\frac{6}{5}}\right) 
\le \frac{1}{2 C_1} \min \{ p-2, (p-1) \omega \},
\]
then we have
\[
(2p-1)I(u) \ge \frac{\min \{ p-2, (p-1) \omega \}}{2 C_1} \| u \|^2,
\]
from which we conclude.
\end{proof}

\begin{lemma} \label{lem:3.4}
Suppose that $2<p<5$. 
There exists $\rho_0>0$ such that if 
\[
e^2 \left(\|\rho\|_{\frac{6}{5}}+\|x \cdot \nabla \rho \|_{\frac{6}{5}}
+\| x \cdot (D^2 \rho x ) \|_{\frac{6}{5}}\right) \le \rho_0,
\]
then $\CM$ is a co-dimension one manifold.
\end{lemma}

\begin{proof}
By Lemma \ref{lem:3.2}, it suffices to show that $J'(u) \ne 0$ if $u \in \CM$.
Suppose by contradiction that $u \in \CM$ satisfies $J'(u)=0$.
Then by Lemma \ref{lem:2.1} (i), one finds that 
$u$ is a weak solution of the problem:
\begin{equation} \label{eq:3.5} 
-3 \Delta u + \omega u - (2 p-1) |u|^{p-1} u
+3 e^2 S_0(u) u -e^2 S_1 u+ e^2 S_2 u=0.
\end{equation}
Especially we have
\begin{equation} \label{eq:3.6}
0= J^{\prime}(u) u 
=3 A(u)+ \omega B(u) -(2p-1) C(u)+12 e^2 D(u) 
-4 e^2 E_1(u)+2 e^2 E_2(u).
\end{equation}
Furthermore multiplying $x \cdot \nabla u$ by \ef{eq:3.5},
using Lemma \ref{lem:2.1} (ii) and arguing as in Lemma \ref{lem:2.5},
one also finds that
\begin{equation} \label{eq:3.7}
\begin{aligned}
0=Q(u) &:= \frac{3}{2} A(u) +\frac{3\omega}{2} B(u)-\frac{3(2p-1)}{p+1} C(u)
+15 e^2 D(u) \\
&\quad -10 e^2 E_1(u)+7 e^2 E_2(u)+e^2 E_3(u).
\end{aligned}
\end{equation}

Now from \ef{eq:2.3}, \ef{eq:3.1}, \ef{eq:3.6} and \ef{eq:3.7},
we obtain the following system of equations:
\begin{equation} \label{eq:3.8}
\begin{pmatrix}
I-2 e^2 E_1 \\
2 e^2 E_1-e^2 E_2 \\
4 e^2 E_1-2 e^2 E_2 \\
10 e^2 E_1- 7e^2 E_2-e^2 E_3
\end{pmatrix}
=
\begin{pmatrix}
\frac{1}{2} & \frac{\omega}{2} & -\frac{1}{p+1} & e^2 \smallskip \\
\frac{3}{2} & \frac{\omega}{2} & -\frac{2 p-1}{p+1} & 3 e^2 \\
3 & \omega & -(2 p-1) & 12 e^2 \\
\frac{3}{2} & \frac{3 \omega}{2} & -\frac{3(2p-1)}{p+1} & 15 e^2
\end{pmatrix}
\begin{pmatrix}
A \\
B \\
C \\
D
\end{pmatrix}.
\end{equation}
Solving \ef{eq:3.8} with the aid of Mathematica \cite{Wol},
it follows that
\[
D(u) = \frac{2 p-1}{24(p-2)}
\left\{ 16 e^2 E_1(u)-7 e^2 E_2(u)-e^2 E_3(u) - 3 I(u) \right\}.
\]
By Lemma \ref{lem:2.2} and Lemma \ref{lem:3.3}, we find that
\[
0 \le D(u) \le -\frac{2p-1}{8(p-2)} \alpha_1 \| u\|^2 
+C e^2 \left(\| \rho \|_{\frac{6}{5}}+ \|x \cdot \nabla \rho \|_{\frac{6}{5}}
+\| x \cdot (D^2 \rho x ) \|_{\frac{6}{5}} \right) \| u\|^2,
\]
from which we arrive at a contradiction provided that
$e^2(\| \rho \|_{\frac{6}{5}}+ \|x \cdot \nabla \rho \|_{\frac{6}{5}}
+\| x \cdot (D^2 \rho x ) \|_{\frac{6}{5}})$ is sufficiently small.
\end{proof}

By Lemma \ref{lem:3.4}, we are able to apply the method of Lagrange multiplier,
which yields that if $u$ is a nontrivial critical point of $I|_{\CM}$,
there exists $\mu \in \R$ such that
\begin{equation} \label{eq:3.9}
I'(u)- \mu J'(u) =0.
\end{equation} 

\begin{lemma} \label{lem:3.5}
Suppose that $2<p<5$.
There exists $\rho_0>0$ such that if
\[
e^2(\| \rho \|_{\frac{6}{5}}+ \|x \cdot \nabla \rho \|_{\frac{6}{5}}
+\| x \cdot (D^2 \rho x ) \|_{\frac{6}{5}}) \le \rho_0,
\]
then it holds that $\mu=0$, that is, 
the set $\CM$ is a natural constraint for the functional $I$.
\end{lemma}

\begin{proof}
First we have from \ef{eq:2.9}, \ef{eq:3.6} and \ef{eq:3.9} that
\begin{align} \label{eq:3.10}
0 &= I^{\prime}(u) u-\mu J^{\prime}(u)u 
=N(u)-\mu J^{\prime}(u)u \\
&= (1-3 \mu) A(u) +(1-\mu) \omega B(u)-(1-(2 p-1) \mu) C(u)
+(4-12 \mu) e^2 D(u) \notag \\
&\quad +(4+4 \mu) e^2 E_1(u)-2 \mu e^2 E_2(u). \notag
\end{align}
Furthermore from \ef{eq:2.10} and \ef{eq:3.7}
and arguing as in Lemma \ref{lem:2.5}, one also finds that
\begin{align} \label{eq:3.11}
0 &= P(u) - \mu Q(u) \\
&=\frac{1-3 \mu}{2} A(u)+\frac{3(1-\mu )\omega }{2} B(u)
-\frac{3-3 \mu(2 p-1)}{p+1} C(u) +(5-15 \mu) e^2 D(u) \notag \\
&\quad +(10+10 \mu) e^2 E_1(u)-(1+7 \mu ) e^2 E_2(u)-\mu e^2 E_3(u). \notag
\end{align}
Combining \ef{eq:2.3}, \ef{eq:3.1}, \ef{eq:3.10} and \ef{eq:3.11}, 
we arrive at the following system of equations:
\begin{equation} \label{eq:3.12}
\begin{pmatrix}
I-2 e^2 E_1 \\
2 e^2 E_1 -e^2 E_2 \\
(4+4\mu) e^2 E_1 -2 \mu e^2 E_2 \\
(10+10 \mu) e^2 E_1-(1+7 \mu) e^2 E_2 -e^2 \mu E_3
\end{pmatrix}
=\Lambda 
\begin{pmatrix}
A \\
B \\
C \\
D
\end{pmatrix},
\end{equation}
where 
\[
\Lambda = 
\begin{pmatrix}
\frac{1}{2} & \frac{\omega}{2} & -\frac{1}{p+1} & e^2 \smallskip \\
\frac{3}{2} & \frac{\omega}{2} & -\frac{2 p-1}{p+1} & 3 e^2 \\
3 \mu-1 & (\mu-1) \omega & 1-(2 p-1) \mu & (12 \mu-4)e^2 \\
\frac{3\mu -1}{2} & \frac{3 (\mu -1) \omega}{2} & -\frac{3(2p-1) \mu-3}{p+1} 
& (15 \mu-5)e^2
\end{pmatrix}.
\]
By a direct calculation with the aid of Mathematica, it follows that
\[
\operatorname{det} \Lambda
=\frac{4(p-2)(p-1) e^2 \omega }{p+1}\mu (3 \mu-1).
\]
If $\mu= \frac{1}{3}$, we can perform 
row operations on the augmented matrix to obtain
\[
\left( \begin{array}{l|c}
&I-2 e^2 E_1 \\
& 2 e^2 E_1 -e^2 E_2 \\
\mbox{\smash{\huge{$\Lambda$}}}& (4+4\mu) e^2 E_1 -2 \mu e^2 E_2 \\
& (10+10 \mu) e^2 E_1-(1+7 \mu) e^2 E_2 -e^2 \mu E_3
\end{array} \right)
\]
\[
\to 
\left( \begin{array}{cccc|c}
\frac{1}{2}&\frac{\omega}{2}&-\frac{1}{p+1}&e^2&I-2e^2E_1  \\
0&-\omega&\frac{4-2p}{p+1}&0&-3I+8e^2 E_1-e^2E_2 \\
0&-\omega&2-p&0&8e^2E_1-e^2E_2 \\
0&0&0&0& 3I+\frac{16}{3}e^2E_1 - \frac{7}{3}e^2E_2
-\frac{1}{3}e^2E_3 
\end{array} \right).
\]
However by Lemma \ref{lem:2.2} and Lemma \ref{lem:3.3}, it holds that
\[
3 I(u)+\frac{16}{3} e^2 E_1(u)-\frac{7}{3} e^2 E_2(u)-\frac{1}{3} e^2 E_3(u)>0
\quad \text{for any} \ u \in \CM
\]
if $e^2 ( \| \rho \|_{\frac{6}{5}} + \|x \cdot \nabla \rho\|_{\frac{6}{5}}
+\| x \cdot (D^2 \rho(x) x) \|_{\frac{6}{5}}) \ll 1$,
which is a contradiction.

On the other hand when $\operatorname{det} \Lambda \ne 0$,
one can solve \ef{eq:3.12} to obtain
\[
0 \le C(u) = \frac{p+1}{4(p-1)(p-2)}
\left\{ 16e^2 E_1(u) - 7e^2 E_2(u) -e^2E_3(u) -3 I(u) \right\}.
\]
This leads a contradiction provided that 
$e^2 ( \| \rho \|_{\frac{6}{5}} + \|x \cdot \nabla \rho\|_{\frac{6}{5}}
+\| x \cdot (D^2 \rho x) \|_{\frac{6}{5}} )$ is sufficiently small,
from which we conclude that $\mu=0$.
\end{proof}

Now let us define the ground state energy level for \ef{eq:1.1} by
\[
m := \inf_{u \in \CS} I(u),
\quad \CS =\left\{ u \in \HT \setminus \{0 \} 
\mid I'(u)=0 \right\}.
\]
By Lemma \ref{lem:3.5}, we are able to prove the following.

\begin{proposition} \label{prop:3.6}
Suppose that $2<p<5$ and assume that
\[
e^2 \left( \| \rho \|_{\frac{6}{5}} +\|x \cdot \nabla \rho \|_{\frac{6}{5}}
+\| x \cdot (D^2 \rho x) \|_{\frac{6}{5}}\right) \le \rho_0
\]
for sufficiently small $\rho_0>0$.

If $\sigma$ defined in \ef{eq:3.4} is attained by some $u_0 \in \HT \setminus \{0\}$,
then $u_0$ is a ground state solution of \ef{eq:1.1}, namely $u_0$ satisfies
\[
\sigma = I(u_0)= m.
\]
\end{proposition}

\begin{proof}
Let $\tilde{u} \in \HT$ be a nontrivial critical point of $I$.
Then by Lemma \ref{lem:2.5}, it follows that $\tilde{u} \in \CM$ and hence
\[
\sigma = \inf_{u \in \CM} I(u) \le I(\tilde{u}).
\]
Taking infimum over $\CS$, we find that $\sigma \le m$.

On the other hand by Lemma \ref{lem:3.5}, if $\sigma$ is achieved by $u_0 \in \CM$,
one has $I'(u_0)=0$.
This implies that $u_0 \in \CS$ and
\[
m \le \inf_{u \in \CS} I(u) \le I(u_0) = \sigma,
\]
from which we conclude that $m= \sigma$.
\end{proof}

By Proposition \ref{prop:3.6}, 
it suffices to investigate the attainability of $\sigma$.
To this aim, we next establish the following energy estimate,
which is a key tool in this paper 
and will be also used to prove the equivalence between
two types of ground state solutions in Section 5.

\begin{lemma} \label{lem:3.7}
Suppose that $2<p<5$ and take $T \ge 4$ so that $T^{2p-4} \ge 3$.
There exist $\alpha= \alpha(T)>0$ and $\beta= \beta(T) >0$
independent of $e$, $\rho$, $\la$ such that the following estimates hold:
For any $u \in H^1(\R^3, \C)$,
\begin{align} \label{3.7-1}
&I(u)-I(u_\la)-\frac{1-\la^3}{3} J(u) \\
&\ge \frac{(1-\la)^2 \omega}{3} \|u \|_2^2
+\frac{\alpha}{6(p+1)}(1-\la)^2 \| u \|_{p+1}^{p+1} \notag \\
&\quad -\beta(1-\la)^2 e^2
\left( \| \rho \|_{\frac{6}{5}} +\|x \cdot \nabla \rho \|_{\frac{6}{5}}
+\| x \cdot (D^2 \rho x) \|_{\frac{6}{5}} \right)
\left( \|u\|_2^2 + \| u \|_{p+1}^2\right) \quad \text{for} \ 0 \le \la \le T, \notag
\end{align}
\begin{align} \label{3.7-2}
&I(u)-I(u_\la)-\frac{1-\la^3}{3} J(u) \\
&\ge \frac{(1-\la)^2 \omega}{6} \| u \|_2^2 +\frac{\la^3 \omega}{3 T} \| u \|_2^2
+\frac{\alpha}{12(p+1)}(1-\la)^2 \| u \|_{p+1}^{p+1}
+\frac{\alpha}{6(p+1)} \la^3 \| u \|_{p+1}^{p+1} \notag \\
&\quad -\beta \la^3 e^2 
\left( \| \rho \|_{\frac{6}{5}} + \| x \cdot \nabla \rho \|_{\frac{6}{5}} 
+\| x \cdot (D^2 \rho x) \|_{\frac{6}{5}} \right)
\left( \|u\|_2^2 + \| u \|_{p+1}^2\right) \quad \text{for} \ \la \ge T. \notag
\end{align}

Assume further that $\| u \|_{p+1} \ge \delta$ for some $\delta>0$
independent of $e$, $\rho$.
There exists $\rho_0>0$ 
independent of $e$, $\rho$, $\la$ such that if
\[
e^2 \left( \| \rho \|_{\frac{6}{5}} +\|x \cdot \nabla \rho \|_{\frac{6}{5}}
+\| x \cdot (D^2 \rho x) \|_{\frac{6}{5}} \right) \le \rho_0,
\]
then the following estimate holds:
\begin{align} \label{eq:3.13}
&I(u)-I(u_\la) - \frac{1-\la^3}{3} J(u) 
\ge \frac{(1-\la)^2 \omega}{6} \| u \|_2^2
+\frac{\alpha \delta^{p-1}}{12(p+1)}(1-\la)^2 \| u \|_{p+1}^2 
\quad \text { for all } \ \la >0.
\end{align}
\end{lemma}

\begin{proof}
The proof consists of four steps.

\smallskip
\noindent
\textbf{Step 1} (Transformation of $I(u)-I(u_\la)$): 
First we observe from \ef{eq:3.2} that
\[
\begin{aligned}
I(u)-I(u_\la) 
&= \frac{1-\la^3}{2} A(u) +\frac{(1-\la) \omega}{2} B(u)-\frac{1-\la^{2p-1}}{p+1} C(u) 
+(1-\la^3 ) e^2 D(u) \\
&\quad +2 e^2 E_1(u) +\frac{e^2 \la^{-1}}{2} \intR S_0(u) \rho(\la^{-1} x) \,dx.
\end{aligned}
\]
Transforming \ef{eq:3.1}, one also has
\[
\frac{1}{2} A(u)+e^2 D(u)
=\frac{1}{3} J(u) -\frac{\omega}{6} B(u)
+\frac{2p-1}{3(p+1)} C(u) 
+\frac{2}{3} e^2 E_1(u) -\frac{1}{3} e^2 E_2(u),
\]
from which we deduce that
\begin{align*} 
I(u)-I(u_\la) &= \frac{1-\la^3}{3} J(u) +\frac{(1-\la)^2(\la+2)\omega }{6} B(u) \\
&\quad +\frac{1}{3(p+1)} \left( 3 \la^{2 p-1}-(2p-1) \la^3+2 p-4 \right) C(u) \notag \\
&\quad +\frac{8-2 \la^3}{3} e^2 E_1(u)
-\frac{1-\la^3}{3} e^2 E_2(u)
+\frac{e^2 \la^{-1}}{2} \int_{\R^3} S_0(u) \rho(\la^{-1} x) \,d x. \notag
\end{align*}
Moreover putting
\[
\begin{aligned}
R(\la,u) &:= 
\frac{8-2 \la^3}{3} e^2 E_1(u) -\frac{1-\la^3}{3} e^2 E_2(u)
+\frac{e^2 \la^{-1}}{2} \int_{\mathbb{R}^3} S_0(u) \rho(\la^{-1} x) \,dx \\
&= e^2 \int_{\mathbb{R}^3} S_0(u) M(\la,x) \,dx, \\
M(\la,x) &:= \frac{\la^3-1}{6} \big( \rho(x) + x \cdot \nabla \rho(x) \big)
- \frac{\rho(x)}{2} +\frac{\rho(\la^{-1}x)}{2 \la},
\end{aligned}
\]
we arrive at 
\begin{align} \label{eq:3.14}
I(u) - I(u_\la) &= 
\frac{1-\la^3}{3} J(u) + \frac{(1-\la)^2(\la +2) \omega}{6} B(u)  \\
&\quad +\frac{1}{3(p+1)} \left(3 \la^{2p-1}- (2p-1) \la^3+2 p-4 \right) C(u)
+R(\la, u). \notag
\end{align}

\noindent
\textbf{Step 2} (Evaluation of coefficients): 
Now for $T \ge 4$, one has
\[
\begin{aligned}
\frac{(1-\la)^2(\la +2)}{6} 
&= \frac{(1-\la)^2}{3}+\frac{\la^3}{6}-\frac{\la^2}{3}+\frac{\la}{6} \\
&\ge \frac{(1-\la)^2}{3} +\frac{2 \la^3}{3 T}-\frac{\la^3}{3 T}
=\frac{(1-\la)^2}{3}+\frac{\la^3}{3 T} \quad \text { for } \ \la \ge T,
\end{aligned}
\]
from which we get
\begin{equation} \label{3.7-3}
\frac{(1-\la)^2(\la+2)}{6} \ge
\begin{cases}
\frac{(1-\la)^2}{3} &\text { for } \ 0 \le \la \le T, \\
\frac{(1-\la)^2}{6}+\frac{\la^3}{3 T} &\text { for } \ \la \ge T.
\end{cases}
\end{equation}

Next we claim that 
\begin{equation} \label{3.7-4}
\frac{1}{3(p+1)} \left( 3 \la^{2 p-1}-(2 p-1) \la^3+2 p-4 \right) 
\ge \begin{cases}
\frac{\alpha}{6(p+1)}(1-\la)^2 &\text { for } 0 \le \la \le T, \\
\frac{\alpha}{12 (p+1)} (1-\la)^2 +\frac{\alpha}{6(p+1)} \la^3 
&\text{ for } \ \la \ge T.
\end{cases}
\end{equation}
For this purpose, we put 
$g(\la):= 3 \la^{2p-1} - (2p-1) \la^3 + 2p-4$.
For $\la \ge T$, one has
\[
g(\la) \ge \la^3 \left(3 \la^{2 p-4}-(2 p-1) \right) 
\ge \la^3 \left( 3 T^{2 p-4}-(2 p-1)\right).
\]
Since $2<p<5$ and $T^{2p-4} \ge 3$, 
it holds $\frac{2p-1}{3} < 3 \le T^{2p-4}$.
Taking $C_1 = 3T^{2p-4}-(2p-1)>0$, 
we get
\begin{equation} \label{3.7-5}
g(\la) \ge C_1 \la^3 \quad \text{for} \ \la \ge T.
\end{equation}
Now we take $0<\tau<1$ so that $(1-\tau)^{2p-4} = \frac{p}{2(p-1)}$,
which is possible because $p>2$.
For $0 \le \la \le 1-\tau$, it holds that
$g'(\la) = 3(2p-1) \la^2 (\la^{2p-4}-1) <0$ and thus
\begin{equation} \label{3.7-6}
g(\la) \ge g(1-\tau) \ge  g(1-\tau)(1-\la)^2 \quad \text {for} \ 0 \le \la \le 1-\tau.
\end{equation}
Next we show that 
\begin{equation} \label{3.7-7}
g(\la) \ge C_2 (1-\la)^2 \quad \text{for} \ \la \ge 1+\tau 
\ \text{and some} \ C_2>0.
\end{equation}
Letting $h(\la) := g(\la)-C_2(1-\la)^2$, one finds that
\[
\begin{aligned}
h^{\prime}(\la) &= 3(2 p-1) \la^2 \left( \la^{2 p-4}-1 \right) -2 C_2 \la+2 C_2 
\ge \la^2 \left( 3(2 p-1) (\la^{2 p-4}-1) -2 C_2 \right) \\
&\ge \la^2 \left( 3(2 p-1) \left( (1+\tau)^{2p-4}-1 \right) -2 C_2\right) 
\quad \text{for} \ \la \ge 1+\tau.
\end{aligned}
\]
If $0< C_2 \le \frac{3}{2} (2p-1) \left( (1+\tau)^{2p-4} -1 \right)$, 
then $h'(\la) \ge 0$ on $[1+\tau, \infty)$.
Thus taking 
\[
C_2 = \min \left\{ \frac{g(1+\tau)}{\tau^2}, 
\frac{3}{2}(2p-1) \left( (1+\tau)^{2 p-4}-1 \right) \right\},
\]
we deduce that $h'(\la) \ge 0$ and $h(\la) \ge h(1+\tau) \ge 0$,
which shows \ef{3.7-7}.

Finally since $g(1)=g'(1)=0$ and $g''(\la)=6(2p-1) \left( (p-1)\la^{2p-4} -1 \right)$, 
we have by the Taylor theorem that
\[
g(\la)= \frac{1}{2} g''(\xi) (1-\la)^2 \quad \text{for some $\xi$ between $\la$ and $1$}.
\]
When $1-\tau \le \la \le 1+\tau$, it follows that 
$\la^{2p-4} \ge (1-\tau)^{2p-4} = \frac{p}{2(p-1)}$, 
from which we obtain
$g''(\la) \ge 6 (2p-1) \left( \frac{p}{2} -1 \right)$ and
\begin{equation} \label{3.7-8}
g(\la)\ge \frac{3}{2}(2p-1)(p-2)(1-\la)^2 \quad \text{for} \ 1-\tau \le \la \le 1+\tau.
\end{equation}
Combining \ef{3.7-5}-\ef{3.7-8} and putting
\[
\alpha := \min \left\{ \frac{3}{2}(p-1)(p-2), g(1-\tau), C_1, C_2 \right\},
\]
we arrive at
\[
g(\la) \ge \begin{cases}
\alpha(1-\la)^2 &\text{for} \ \la \ge 0, \\
\alpha \la^3 &\text{for} \ \la \ge T.
\end{cases}
\]
Moreover since 
\[
\frac{1}{3(p+1)} \left( 3 \la^{2 p-1}-(2p-1) \la^3+2 p-4\right)
= \frac{1}{3(p+1)} g(\la)
= \frac{1}{6(p+1)} g(\la) +\frac{1}{6(p+1)} g(\la),
\]
we obtain \ef{3.7-4}.

\smallskip
\noindent
\textbf{Step 3} (Estimate for $R(\la,u)$): 
First when $\la \ge T$, one find from the assumption \ef{eq:1.6} that
\[
\begin{aligned}
M(\la, x) 
&\ge -\frac{\la^3-1}{6} \big( |\rho (x)|+| x \cdot \nabla \rho(x) | \big)
-\frac{1}{2}|\rho(x)| 
\ge -\frac{\la^3}{6} \big( |\rho(x)|+| x \cdot \nabla \rho(x) | \big)
-\frac{1}{2}|\rho(x)| \\
&\ge -\la^3 \left( \frac{2}{3} | \rho(x)| +\frac{1}{6} |x \cdot \nabla \rho(x)| \right).
\end{aligned}
\]
Thus by Lemma \ref{lem:2.2}, we get
\begin{align} \label{3.7-9}
R(\la, u) &\ge - \la^3 e^2 \int_{\R^3} |S_0(u)| 
\left( \frac{2}{3}|\rho(x)|+\frac{1}{6} |x \cdot \nabla \rho(x)|\right) \,d x \notag \\
&\ge - \la^3 e^2 
\left( \frac{2}{3} \| \rho \|_{\frac{6}{5}}
+\frac{1}{6} \| x \cdot \nabla \rho \|_{\frac{6}{5}} \right)
\| S_0(u) \|_6 \notag \\
&\ge -C_3 \la^3 e^2 \left( \|\rho\|_{\frac{6}{5}}
+\|x \cdot \nabla \rho\|_{\frac{6}{5}}\right)
\left(\| u\|_2^2+\| u\|_{p+1}^2 \right)
\quad \text{for} \ \la \ge T,
\end{align}
where $C_3>0$ is independent of $e$, $\rho$, $\la$.
Next for $0 \le \la \le \frac{1}{2}$, we have
\[
M(\la, x) \ge 
\frac{\la^3-1}{6} \big( |\rho(x)|+|x \cdot \nabla \rho(x)| \big) -\frac{1}{2}|\rho(x)| 
\ge -\frac{2}{3}|\rho(x)|-\frac{1}{6} | x \cdot \nabla \rho(x)|,
\]
from which one concludes that
\begin{align} \label{3.7-10}
R(\la,u) &\ge -e^2 \left( \frac{2}{3}\|\rho\|_{\frac{6}{5}}
+\frac{1}{6} \| x \cdot \nabla \rho \|_{\frac{6}{5}} \right) \| S_0(u) \|_6 \notag \\
&\ge - 4(1-\la)^2 e^2 
\left( \frac{2}{3}\|\rho\|_{\frac{6}{5}}
+\frac{1}{6} \| x \cdot \nabla \rho \|_{\frac{6}{5}} \right) \| S_0(u) \|_6 \notag \\
&\ge - C_4(1-\la)^2 e^2 \left( \|\rho\|_{\frac{6}{5}}
+\| x \cdot \nabla \rho \|_{\frac{6}{5}}\right)
\left(\| u\|_2^2+ \| u\|_{p+1}^2 \right) 
\quad \text{for} \ 0 \le \la \le \frac{1}{2}
\end{align}
for some $C_4>0$ independent of $e$, $\rho$, $\la$.
For $\frac{1}{2} \le \la \le T$, we first observe that
\[
\begin{aligned}
\frac{\partial M}{\partial \la}(\la, x) 
&= \frac{\la^2}{2} \big( \rho(x) +x \cdot \nabla \rho(x) \big)
-\frac{\rho(\la^{-1}x)}{2 \la^2} -\frac{x \cdot \nabla \rho(\la^{-1} x)}{2 \la^3}, \\
\frac{\partial^2 M}{\partial \la^2}(\la, x) 
&= \la \big( \rho(x)+x \cdot \nabla \rho(x) \big)
+\frac{\rho(\la^{-1} x)}{\la^3} 
+\frac{2 x \cdot \nabla \rho (\la^{-1} x)}{\la^4}
+\frac{x \cdot \left( D^2 \rho (\la^{-1}x)x \right)}{2\la^5}.
\end{aligned}
\]
Then for fixed $x \in \R^3$, one finds that 
$M(1,x)= \frac{\partial M}{\partial \la}(1, x)=0$ and
\[
\begin{aligned}
\frac{\partial^2 M}{\partial \la^2}(\la, x) 
&\ge - \la \big( |\rho(x)|+| x \cdot \nabla \rho(x)| \big) \\
&\quad -\frac{1}{\la^3} 
\left( |\rho(\la^{-1}x)|+2 |(\la^{-1} x) \cdot \nabla \rho (\la^{-1} x)|
+\frac{1}{2} |(\la^{-1} x) \cdot D^2 \rho (\la^{-1} x) (\la^{-1} x)| \right) \\
&=: - N(\la,x).
\end{aligned}
\]
By the Taylor theorem, there exists $\xi = \xi(\la,x) \in \left( \frac{1}{2}, T \right)$
such that 
\[
M(\la, x)=\frac{1}{2} \frac{\partial^2 M}{\partial \la^2}(\xi, x)(1-\la)^2 
\ge -\frac{1}{2} N(\xi, x)(1-\la)^2.
\]
Using the H\"older inequality, we deduce that 
\[
R(\la, u) \ge -\frac{1}{2}(1-\la)^2 e^2 \|N(\xi, x) \|_{\frac{6}{5}} \|S_0(u) \|_6
\quad \text{for} \ \frac{1}{2} \le \la \le T.
\]
Moreover since $\frac{1}{2} < \xi < T$, we have
\[
\begin{aligned}
\| N(\xi, x)\|_{\frac{6}{5}} 
&\le \| \xi \rho \|_{\frac{6}{5}} + \| \xi x \cdot \nabla \rho \|_{\frac{6}{5}} 
+ \| \xi^{-\frac{1}{2}} \rho \|_{\frac{6}{5}}
+2 \| \xi^{-\frac{1}{2}} x \cdot \nabla \rho \|_{\frac{6}{5}}
+\frac{1}{2} \| \xi^{-\frac{1}{2}} x (D^2 \rho x) \|_{\frac{6}{5}} \\
&\le T \left( \|\rho \|_{\frac{6}{5}} + \| x \cdot \nabla \rho \|_{\frac{6}{5}}\right) 
+\sqrt{2} \left( \| \rho \|_{\frac{6}{5}} + 2 \| x \cdot \nabla \rho \|_{\frac{6}{5}}
+ \frac{1}{2} \| x \cdot (D^2 \rho x) \|_{\frac{6}{5}} \right).
\end{aligned}
\]
Then by Lemma \ref{lem:2.2}, it follows that 
\begin{equation} \label{3.7-11}
R(\la, u) \ge -C_5 (1 - \la)^2 e^2 
\left( \|\rho\|_{\frac{6}{5}}+\| x \cdot \nabla \rho \|_{\frac{6}{5}}
+\| x \cdot (D^2 \rho x) \|_{\frac{6}{5}} \right)
\left( \| u \|_2^2 + \| u \|_{p+1}^2\right)
\quad \text{for} \ \frac{1}{2} \le \la \le T,
\end{equation}
where $C_5>0$ is independent of $e$, $\rho$, $\la$.
From \ef{3.7-9}-\ef{3.7-11}, letting $\beta= \min \{ C_3, C_4, C_5 \}>0$,
we find that
\begin{align} \label{3.7-12}
&R(\la, u) \\
&\ge 
\begin{cases}
- \beta (1 - \la)^2 e^2 
\left( \|\rho\|_{\frac{6}{5}}+\| x \cdot \nabla \rho \|_{\frac{6}{5}}
+\| x \cdot (D^2 \rho x) \|_{\frac{6}{5}} \right)
\left( \| u \|_2^2 + \| u \|_{p+1}^2\right)
&\text{for} \ 0 \le \la \le T, \\
- \beta \la^3 e^2 
\left( \|\rho\|_{\frac{6}{5}}+\| x \cdot \nabla \rho \|_{\frac{6}{5}}
+\| x \cdot (D^2 \rho x) \|_{\frac{6}{5}} \right)
\left( \| u \|_2^2 + \| u \|_{p+1}^2\right)
&\text{for} \ \la \ge T.
\end{cases} \notag
\end{align}

\noindent
\textbf{Step 4} (Conclusion): 
Now from \ef{eq:3.14}, \ef{3.7-3}, \ef{3.7-4} and \ef{3.7-12},
we can see that \ef{3.7-1} and \ef{3.7-2} hold.

Finally suppose that $\| u \|_{p+1} \ge \delta$.
If 
\[
e^2 \left( \| \rho \|_{\frac{6}{5}}+\|x \cdot \nabla \rho\|_{\frac{6}{5}}
+\|x (D^2 \rho x ) \|_{\frac{6}{5}} \right) 
\le \frac{1}{2 \beta} \min 
\left\{ \frac{2 \omega}{3 T}, \frac{\alpha \delta^{p-1}}{6(p+1)} \right\},
\]
using \ef{3.7-1} and \ef{3.7-2}, we obtain
\[
I(u)-I(u_\la)-\frac{1-\la^3}{3} J(u) 
\ge \frac{(1-\la)^2\omega }{6}\| u\|_2^2
+\frac{\alpha \delta^{p-1}}{12(p+1)}(1-\la)^2 \| u\|_{p+1}^2
\quad \text{for all} \ \la >0.
\]
This finishes the proof.
\end{proof}

\begin{remark} \label{rem:3.8}
Letting 
\[
F(\la, x) :=
\frac{1}{\la^4} \left( \rho(\la^{-1} x) + (\la^{-1} x) \cdot \nabla \rho (\la^{-1} x)\right),
\]
we find that $M(\la,x)$ can be written as
\[
\begin{aligned}
M(\la, x) & =\frac{\la^3-1}{6} \big( \rho(x)+x \cdot \nabla \rho(x) \big)
-\frac{\rho(x)}{2} +\frac{\rho(\la^{-1} x)}{2 \la} \\
&= \int_1^\la \frac{s^2}{2} \big( \rho(x)+x \cdot \nabla \rho(x) \big) \,d s
+\frac{1}{2} \int_1^{\la} \frac{d}{ds} \left( s^{-1} \rho (s^{-1} x ) \right) \,d s \\
&= \frac{1}{2} \int_1^\la s^2 \big( F(1, x)-F(s,x) \big) \,d s .
\end{aligned}
\]
Thus if $F(\la,x)$ is non-increasing with respect to $\la$ for every $x \in \R^3$,
it holds that $M(\la,x) \ge 0$ and hence
$R(\la,u) \ge 0$ for all $\la >0$,
which is a same situation to \cite{TC}.
However we cannot expect that $F(\la,x)$ is non-increasing in $\la$
for doping profiles.

Indeed by a direct calculation, $F(\la,x)$ is non-increasing in $\la$ 
if $\rho$ satisfies
\begin{equation} \label{eq:3.18}
4 \rho (x)+6 x \cdot \nabla \rho(x)
+x \cdot \left( D^2 \rho(x) x \right) 
\ge 0 \quad \text { for all } x \in \R^3.
\end{equation}
When $\rho$ is a Gaussian function $e^{-\alpha |x|^2}$ for $\alpha >0$,
one finds that
\[
2 \rho(x)+3 x \cdot \nabla \rho(x)
+\frac{1}{2} x \cdot \left(D^2 \rho(x) x\right) 
= \left(2 \alpha^2 |x|^4 -7 \alpha |x|^2+2 \right) e^{-\alpha |x|^2}.
\]
Hence no matter how we choose $\alpha$,
\ef{eq:3.18} fails to hold near the inflection point.
Moreover if we consider $\rho(x)= \frac{1}{1+|x|^r}$, we see that
\[
2\rho(x)+ 3 x \cdot \nabla \rho(x)
+\frac{1}{2} x \cdot \left(D^2 \rho(x) x\right) 
=\frac{1}{2 (1+ |x|^r)^3}
\left\{ (r-1)(r-4) |x|^{2 r} - (r^2+5 r-8 ) |x|^r+4 \right\}.
\]
Then it follows that
\[
\begin{aligned}
r>4 &\quad \Rightarrow \quad 
\text{\ef{eq:3.18} fails to hold near the inflection point}, \\
1 \leq r \leq 4 &\quad \Rightarrow \quad
\text{\ef{eq:3.18} fails to hold for large $|x|$}, \\
0<r<1 &\quad \Rightarrow \quad \rho \notin L^{\frac{6}{5}}(\R^3).
\end{aligned}
\]
In this sense, the assumption \ef{eq:1.5} and \ef{eq:3.18} 
seem to be inconsistent.
\end{remark}

Now using Lemma \ref{lem:3.7}, we can show the following.

\begin{lemma} \label{lem:3.9}
Suppose that $2<p<5$ and assume that
\[
e^2 \left( \| \rho \|_{\frac{6}{5}} +\| x \cdot \nabla \rho \|_{\frac{6}{5}}
+\| x \cdot (D^2 \rho x ) \|_{\frac{6}{5}} \right) \le \rho_0
\]
for sufficiently small $\rho_0>0$.

Then for any $u \in \HT \setminus \{0 \}$, 
there exists a unique $\la_u>0$ such that $u_{\la_u} \in \CM$.
Especially for $u \in \CM$,
\[
u_\la \in \CM \quad \text{if and only if} \quad \la_u=1. 
\]
\end{lemma}

\begin{proof}
By Lemma \ref{lem:3.1}, we know that there exists $\la_u>0$,
which is a maximum point of $f(\la)$ defined in \ef{eq:3.2},
such that $u_{\la_u} \in \CM$.
Thus it suffices to show that $f(\la)$ has a unique critical point for $\la>0$.

Now suppose by contradiction that there exist
$0<\la_1<\la_2$ such that $f'(\la_1)=f'(\la_2)=0$.
As we have observed in the proof of Lemma \ref{lem:3.1},
it holds that $J(u_\la)=\la f'(\la)$, which yields that
\begin{equation} \label{eq:3.19}
J(u_{\la_1})=J(u_{\la_2})=0 .
\end{equation}
Since $u_{\la_1}$, $u_{\la_2} \in \CM$, we can use \ef{eq:3.3} to show that
$\| u_{\la_1} \|_{p+1} \ge \delta_1$ and $\| u_{\la_2} \|_{p+1} \ge \delta_1$.
Then we are able to apply \ef{eq:3.13} if
$e^2 \left( \| \rho \|_{\frac{6}{5}} + \| x \cdot \nabla \rho \|_{\frac{6}{5}}
+ \| x \cdot (D^2 \rho x) \|_{\frac{6}{5}} \right)$ is sufficiently small.

Applying \ef{eq:3.13} with $u=u_{\la_1}$ and $\la= \frac{\la_2}{\la_1}$, we have
\[
I(u_{\la_1}) \ge I(u_{\la_2})
+\frac{\la_1^3-\la_2^3}{3 \la_1^3} J(u_{\la_1})
+\frac{(\la_1-\la_2)^2\omega}{6 \la_1^2} \| u_{\la_1} \|_2^2 
+\frac{\alpha \delta_1^{p-1} (\la_1-\la_2)^2}{12(p+1) \la_1^2} \| u_{\la_1} \|_{p+1}^2. 
\]
Similarly one gets
\[
I(u_{\la_2}) \ge I(u_{\la_1})
+\frac{\la_2^3-\la_1^3}{3 \la_2^3} J(u_{\la_2})
+\frac{(\la_2-\la_1)^2\omega}{6 \la_2^2} \| u_{\la_2} \|_2^2 
+\frac{\alpha \delta_1^{p-1} (\la_2-\la_1)^2}{12(p+1) \la_2^2} \| u_{\la_2} \|_{p+1}^2.
\]
Thus from \ef{eq:3.19}, it follows that 
\[
0 \ge \frac{\omega}{6} (\la_1-\la_2 )^2
\left( \frac{1}{\la_1^2} \| u_{\la_1} \|_2^2 +\frac{1}{\la_2^2} \| u_{\la_2} \|_2^2\right) 
+ \frac{\alpha \delta_1^{p-1}}{12(p+1)} (\la_1-\la_2 )^2
\left( \frac{1}{\la_1^2} \| u_{\la_1} \|_{p+1}^2+ \frac{1}{\la_2^2} \| u_{\la_2} \|_{p+1} ^2\right)
>0, 
\]
from which arrive at a contradiction and 
conclude that $f(\la)$ has a unique critical point.
\end{proof}

\begin{remark} \label{rem:3.10}
By Proposition \ref{prop:3.6} and Lemma \ref{lem:3.9}, 
we can obtain the following minimax characterization of $m$:
\[
m= \inf_{u \in \HT \setminus \{0 \}} \max_{\la>0} I(u_\la).
\] 

\end{remark}

\section{Existence of a ground state solution for $2<p<5$}

In this section, we establish the existence of a ground state solution of \ef{eq:1.1}.
For this purpose, we define the energy functional $I_{\infty}$ 
associated with \ef{eq:1.3} by 
\[
I_{\infty}(u) := \frac{1}{2} A(u) +\frac{\omega}{2} B(u) 
-\frac{1}{p+1} C(u)+e^2 D(u) \quad \text{for} \ u \in \HT.
\]
Indeed if $\rho(x) \to 0$ as $|x| \to \infty$, 
$I_{\infty}$ can be seen as a functional at infinity.
We define the ground state energy corresponding to \ef{eq:1.3} by
\[
m_{\infty} := \inf_{u \in \CS_{\infty}} I_{\infty}(u), \quad
\CS_{\infty} := \left\{ u \in \HT \mid I_{\infty}^{\prime}(u)=0 \right\}.
\]
Let us denote by $J_{\infty}$ the Nehari-Pohozaev functional for $I_{\infty}$, that is,
\[
J_{\infty}(u) := \frac{3}{2} A(u)+\frac{\omega}{2} B(u)
-\frac{2p-1}{p+1} C(u)+3e^2 D(u). 
\]
We also set
\[
\sigma_{\infty} := \inf_{u \in \CM_{\infty}} I_{\infty}(u), \quad 
\CM_{\infty} := \left\{ u \in \HT \setminus \{ 0 \} \mid 
J_{\infty}(u)=0 \right\}.
\]
Then by the result in \cite{AP}
with minor modifications considering $u \in \HT$, 
it follows that $\sigma_{\infty}= m_{\infty}$
and $\sigma_{\infty}$ is achieved by some $u_{\infty} \in \HT \setminus \{0 \}$
for $2<p<5$ and any $e>0$.
Moreover arguing as Lemma \ref{lem:3.7} (see also \cite[Corollary 3.3]{TC}),
we have
\begin{equation} \label{eq:4.1}
I_{\infty}(u) \ge I_{\infty}(u_\la) +\frac{1-\la^3}{3} J_{\infty}(u) \quad 
\text {for all } u \in \HT \ \ \text{and} \ \ \la>0,
\end{equation}
where $u_\la(x)=\la^2u(\la x)$.

To prove the attainability of $\sigma$, we need the following lemma.
Note that this is the only part where we crucially require \ef{eq:1.6}.

\begin{lemma} \label{lem:4.1}
Suppose that $2<p<5$. Assume further \ef{eq:1.5} and \ef{eq:1.6}.
Then it follows that
\[
\sigma < \sigma_{\infty} \quad \text{and} \quad m < m_{\infty}.
\]
\end{lemma}

\begin{proof}
First we observe from \ef{eq:1.6} that
\[
I(u ) < I_{\infty}(u) \quad \text { for all } u \in \HT\setminus \{ 0 \}
\]
because $E_1(u) <0$ by \ef{eq:2.2}.

Now let $u_{\infty}$ be a ground state solution for \ef{eq:1.3}.
By Lemma \ref{lem:3.1}, 
there exists $\la_{\infty}>0$ such that 
$\tilde{u}_{\infty}(x):=\la_{\infty}^2 u_{\infty}(\la_{\infty} x) \in \CM$.
Then from \ef{eq:4.1} and $J_{\infty}(u_{\infty})=0$, we obtain
\[
\sigma \le I(\tilde{u}_{\infty}) < I_{\infty} (\tilde{u}_{\infty} ) 
\le I_{\infty}(u_{\infty} )= \sigma_{\infty},
\]
which ends the proof.
\end{proof}

\begin{remark} \label{rem:4.2}
When we drop \ef{eq:1.6}, we cannot determine the sign of
\[
E_1(u) = \frac{1}{4} \intR S_1(x) |u|^2 \,dx
= - \frac{1}{32 \pi} \intR \intR \frac{\rho(y) |u(x)|^2}{|x-y|} \,dx \,dy,
\]
causing that the estimate $\sigma < \sigma_{\infty}$ is completely nontrivial.
As we will see in the proof of Proposition \ref{prop:4.2} below,
this estimate is a key ingredient of recovering the compactness of minimizing sequences for $\sigma$.

If we could assume that $S_1(x)$ decays exponentially at infinity,
we still have a possibility to recover the compactness as in \cite{AW2}.
Under the assumption \ef{Ass}, one can show that 
$S_1(x) \to 0$ as $|x| \to \infty$; see \cite[Lemma 2.4]{CW4}.
However no matter how fast $\rho(x)$ decays at infinity,
$S_1(x)$ cannot decay faster than $\frac{1}{|x|}$,
yielding that the assumption in \cite{AW2} is never satisfied.
For the moment, we don't know whether Lemma \ref{lem:4.1} holds without \ef{eq:1.6}.
\end{remark}

Now we are ready to prove the following result.

\begin{proposition} \label{prop:4.2}
Suppose that $2<p<5$. Assume further \ef{eq:1.5}, \ef{eq:1.6} and 
\[
e^2 \left(\|\rho\|_{\frac{6}{5}}+\|x \cdot \nabla \rho\|_{\frac{6}{5}}
+\| x \cdot (D^2 \rho x ) \|_{\frac{6}{5}}\right) \le \rho_0
\]
for sufficiently small $\rho_0>0$.

Then there exists $u_0 \in \HT \setminus \{0 \}$ such that
\[
I(u_0)= \sigma \quad \text{and} \quad J(u_0)=0.
\]
\end{proposition}

\begin{proof}
Let $\{ u_n \} \subset \CM$ be a minimizing sequence for $\sigma$, that is,
\begin{equation} \label{eq:4.2}
I(u_n) = \sigma +o(1) \quad \text {and} \quad J(u_n) =0.
\end{equation}
Then by Lemma \ref{lem:3.3}, it follows that
\[
\sigma +o(1) = I(u_n) \ge \alpha_1 \| u\|^2,
\]
which implies that $\| u_n \|$ is bounded.
Thus passing to a subsequence, 
we may assume that $u_n \rightharpoonup u_0$ in $\HT$
for some $u_0 \in \HT$.
We divide the proof into two steps.

\smallskip
\noindent
\textbf{Step 1}: We claim that $u_0 \not\equiv 0$.

Suppose by contradiction that $u_0 \equiv 0$ so that 
$u_n \rightharpoonup 0$ in $\HT$.
Then by Lemma \ref{lem:2.4}, it follows that
\[
E_1(u_n) \rightarrow 0 \quad \text { and }\quad 
E_2(u_n) \rightarrow 0 \quad \text {as} \ n \to \infty.
\]
Since we can write
\[
I(u)=I_{\infty}(u)+2e^2E_1(u) \quad \text {and} \quad 
J(u)=J_{\infty}(u) -2e^2 E_1(u)+e^2 E_2(u),
\]
we have from \ef{eq:4.2} that 
\begin{equation} \label{eq:4.3}
I_{\infty}(u_n) \rightarrow \sigma \quad \text{and} \quad 
J_{\infty}(u_n) \rightarrow 0 \quad \text{as} \ n \rightarrow \infty.
\end{equation}
Arguing as in Lemma \ref{lem:3.2}, we also deduce that $\| u_n \|_{p+1} \ge \delta$
for some $\delta>0$ independent of $e$, $\rho$ and $n \in \N$.

Now we apply the concentration compactness principle \cite{L1, Wil}
to show that there exist $\tilde{\delta}>0$ and $\{ y_n \} \subset \R^3$ such that
\[
\int_{B_1(y_n)} |u_n(x)|^{p+1} \,d x \geq \tilde{\delta}.
\]
Letting $\hat{u}_n(x) := u_n(x+y_n)$, we have \[
\| \hat{u}_n \| = \| u_n \| \quad \text{and} \quad
\int_{B_1(0)} |\hat{u}_n(x)|^{p+1} \,d x \ge \tilde{\delta}.
\]
Especially, there exists $\hat{u} \in \HT \setminus \{0 \}$ such that
$\hat{u}_n \rightharpoonup \hat{u}$ in $\HT$.
Moreover from \ef{eq:4.3}, one also finds that
\begin{equation} \label{eq:4.4}
I_{\infty}(\hat{u}_n)= I_{\infty}(u_n) = \sigma +o(1) \quad \text{and} \quad 
J_{\infty}(\hat{u}_n)= J_{\infty}(u_n) = o(1).
\end{equation}

Let us put $v_n:= \hat{u}_n- \hat{u}$. 
Then by the Brezis-Lieb lemma \cite{BrLi} and Lemma \ref{lem:2.3}, it follows that
\[
I_{\infty}(\hat{u}_n) = I_{\infty}(\hat{u})+I_{\infty}(v_n)+o(1) \quad \text{and} \quad 
J_{\infty}(\hat{u}_n) = J_{\infty}(\hat{u})+J_{\infty}(v_n)+o(1).
\]
We also define
\[
K_{\infty}(u) := I_{\infty}(u) -\frac{1}{3} J_{\infty}(u) 
= \frac{\omega}{3} B(u) +\frac{2(p-2)}{3(p-1)} C(u).
\]
Then from \ef{eq:4.4}, one finds that
\begin{equation} \label{eq:4.5}
K_{\infty}(v_n) = \sigma -K_{\infty}(\hat{u})+o(1) \quad \text{and} \quad
J_{\infty}(v_n)= -J_{\infty}(\hat{u})+o(1).
\end{equation}
If there exists a subsequence $\{ v_{n_j} \} \subset \{ v_n \}$ such that
$v_{n_j}=0$, then passing to a limit along this subsequence,
it holds from \ef{eq:4.5} that 
$K_{\infty}(\hat{u})=\sigma$ and $J_{\infty}(\hat{u})=0$.
This implies that
\[
\sigma_{\infty} \le I_{\infty}(\hat{u}) 
= K_\infty(\hat{u})+\frac{1}{3} J_{\infty}(\hat{u})= \sigma,
\]
which is absurd by Lemma \ref{lem:4.1}.
Thus we may assume that $v_n \ne 0$.
Arguing as Lemma \ref{lem:3.9} 
(see also \cite[Lemma 3.3]{R}, \cite[Lemma 2.4]{TC}), 
there exists a unique $\hat{\la}_n >0$ such that 
$\hat{\phi}_n(x) := (\hat{\la}_n)^2 v_n (\hat{\la}_n x) \in \CM_{\infty}$.

Next we prove that $J_{\infty}(\hat{u}) \le 0$.
Suppose by contradiction that $J_{\infty}(\hat{u})>0$.
Then from \ef{eq:4.5}, it follows that $J_{\infty}(v_n) \le 0$
and hence $\| v_n \|_{p+1} \ge \delta_1$ by Lemma \ref{lem:3.2}.
Using \ef{eq:3.13} with $\rho \equiv 0$ and \ef{eq:4.5}, one deduces that
\[
\begin{aligned}
\sigma-K_{\infty}(\hat{u})+o(1)
= K_{\infty}(v_n) 
&=I_{\infty}(v_n)-\frac{1}{3} J_{\infty}(v_n) \\
&\ge I_{\infty} (\hat{\phi}_n) -\frac{(\hat{\la}_n)^3}{3} J_{\infty}(v_n) \\
&\ge \sigma_{\infty}-\frac{(\hat{\la}_n)^3}{3} J_\infty(v_n) \ge \sigma_{\infty}, 
\end{aligned}
\]
which yields that
\[
\sigma_{\infty} \le \sigma -K_{\infty}(\hat{u}) 
< \sigma_{\infty} -K_{\infty}(\hat{u}) \quad \text{and thus} \quad K_{\infty}(\hat{u})<0.
\]
This is a contradiction to the fact $K_{\infty}(u)>0$ 
for any $u \in \HT \setminus \{0 \}$
and hence $J_{\infty}(\hat{u}) \le 0$.

Now by Lemma \ref{lem:3.9}, there exists a unique $\hat{\la}>0$ such that
$\hat{\phi}(x) := (\hat{\la})^2 \hat{u}(\hat{\la} x) \in \CM_{\infty}$.
Then from \ef{eq:3.13} with $\rho \equiv 0$, \ef{eq:4.4} 
and by the Fatou lemma, we obtain
\[
\begin{aligned}
\sigma &= \lim_{n \rightarrow \infty}
\left\{ I_{\infty}(\hat{u}_n)-\frac{1}{3} J_{\infty} (\hat{u}_n) \right\}
= \lim_{n \rightarrow \infty} K_{\infty}(\hat{u}_n) 
\ge K_{\infty}(\hat{u}) \\
&= I_{\infty}(\hat{u})-\frac{1}{3} J_{\infty}(\hat{u}) \\
&\ge I_{\infty}(\hat{\phi}) -\frac{(\hat{\la})^3}{3} J_{\infty}(\hat{u}) \\
&\ge \sigma_{\infty}-\frac{(\hat{\la})^3}{3} J_{\infty}(\hat{u}) \ge \sigma_{\infty}.
\end{aligned}
\]
This is a contradiction to Lemma \ref{lem:4.1} and hence 
we conclude that $u_0 \not\equiv 0$.

\smallskip
\noindent
\textbf{Step 2}: We prove that $I(u_0)=\sigma$ and $J(u_0)=0$.

Let us define $w_n:= u_n- u_0$.
By Lemma \ref{lem:2.3}, Lemma \ref{lem:2.4} and the Brezis-Lieb lemma, we have
\[
I(u_n)= I(u_0)+I(w_n) +o(1) \quad \text{and} \quad
J(u_n)=J(u_0)+J(w_n) +o(1).
\]
We also put
\begin{align*} 
K(u) &:= I(u)-\frac{1}{3} J(u) \\
&= \frac{\omega}{3} B(u) +\frac{2(p-2)}{3(p+1)} C(u)
+\frac{8}{3} e^2 E_1(u)-\frac{1}{3} e^2 E_2(u)
\quad \text{for} \ u \in \HT. \notag
\end{align*}
Then from \ef{eq:4.2}, it holds that
\begin{equation} \label{eq:4.6}
K(w_n) = \sigma -K(u_0)+o(1) \quad \text{and} \quad
J(w_n) = - J(u_0)+o(1).
\end{equation}
If there exists a subsequence $\{ w_{n_j} \} \subset \{ w_n \}$ such that
$w_{n_j} =0$, then passing to a limit along this subsequence,
we arrive at $K(u_0)=\sigma$ and $J(u_0)=0$. 
This implies that $I(u_0)=\sigma$ and hence we conclude.
Thus we may assume that $w_n \ne 0$.

Next we show that $J(u_0) \le 0$.
Indeed if $J(u_0)>0$, it follows from \ef{eq:4.6} that
$J(w_n) \le 0$ and hence $\| w_n \|_{p+1} \ge \delta_1$ by Lemma \ref{lem:3.2}.
Then by Lemma \ref{lem:3.9}, there exists a unique $\la_n>0$ such that
$\phi_n(x) := \la_n^2 w_n(\la_n x) \in \CM$.
Using \ef{eq:3.13} and \ef{eq:4.6}, we deduce that
\[
\begin{aligned}
\sigma -K(u_0)+o(1) = K(w_n) &= I(w_n)-\frac{1}{3} J(w_n) \\
&\ge I(\phi_n)-\frac{\la_n^3}{3} J(w_n) \\
&\ge \sigma -\frac{\la_n^3}{3} J(w_n) \ge \sigma,
\end{aligned}
\]
yielding that $K(u_0) \le 0$.
However this is a contradiction because $K(u)>0$ 
for any $u \in \HT \setminus \{ 0 \}$ by Lemma \ref{lem:2.2}
provided that $e^2( \| \rho\|_{\frac{6}{5}} + \| x \cdot \nabla \rho \|_{\frac{6}{5}})$
is sufficiently small,
from which we conclude that $J(u_0) \le 0$.

Now using Lemma \ref{lem:3.9} again, 
there exists a unique $\la_0>0$ such that $\phi_0(x) := \la_0^2 u_0(\la_0 x) \in \CM$.
Moreover since $J(u_0) \le 0$, we are able to use \ef{eq:3.13} by Lemma \ref{lem:3.2}.
Then from \ef{eq:4.2}, Lemma \ref{lem:2.4} and by the Fatou lemma,
one finds that
\begin{align} \label{eq:4.7}
\sigma &= \lim_{n \to \infty} \left\{ I(u_n) - \frac{1}{3} J(u_n) \right\}
= \lim_{n \to \infty} K(u_n) \ge K(u_0) \notag \\
&= I(u_0) -\frac{1}{3} J(u_0) \notag \\
&\ge I(\phi_0) -\frac{\la_0^3}{3} J(u_0) \notag \\
&\ge \sigma -\frac{\la_0^3}{3} J(u_0) \ge \sigma, 
\end{align}
from which we get $K(u_0)=\sigma$.
Moreover if $J(u_0)<0$, \ef{eq:4.7} leads a contradiction.
Thus it follows that $J(u_0)=0$ and hence $I(u_0)=\sigma$.
This completes the proof.
\end{proof}

\begin{remark} \label{rem:4.3}
As we have observed in the proof of Proposition \ref{prop:4.2},
\[
\begin{aligned}
K(u) &= \frac{\omega}{3} B(u) + \frac{2(p-2)}{3(p+1)} C(u)
+\frac{8}{3} e^2 E_1(u) -\frac{1}{3} e^2 E_2(u) \\
&= \intR \left\{ 
\frac{\omega}{3}|u(x)|^2 +\frac{2(p-2)}{3(p+1)} |u(x)|^{p+1} 
-\frac{e^2}{6} S_0(u)(x) \big( \rho(x)+x \cdot \nabla \rho(x) \big) \right\} \,d x
\end{aligned}
\]
plays an important role.
Especially it is natural to apply the concentration compactness principle
to the function
\[
\rho_n(x) := 
\frac{\omega}{3}|u_n(x)|^2 +\frac{2(p-2)}{3(p+1)} |u_n(x)|^{p+1} 
-\frac{e^2}{6} S_0(u_n)(x) \big( \rho(x)+x \cdot \nabla \rho(x) \big)
\]
for a minimizing sequence $\{ u_n \} \subset \CM$.
However we don't know whether $\rho_n(x) \ge 0$ for all $x \in \R^3$,
although one knows that $K(u_n) \ge 0$ 
if $e^2( \| \rho\|_{\frac{6}{5}} + \| x \cdot \nabla \rho \|_{\frac{6}{5}})$
is small by Lemma \ref{lem:2.2}.

If $\rho$ satisfies
\begin{equation} \label{eq:4.8}
\rho(x)+ x \cdot \nabla \rho(x) \le 0 \quad \text {for all} \ x \in \R^3,
\end{equation} 
it follows that $\rho_n \ge 0$ on $\R^3$.
However the condition \ef{eq:4.8} seems to be inconsistent with
\ef{eq:1.5} and \ef{eq:1.6}.
For example, if we consider the Gaussian function 
$\rho(x)=e^{-\alpha |x|^2}$ for $\alpha >0$, we have
\[
\rho(x) +x \cdot \nabla \rho(x) 
=\left( 1-2 \alpha|x|^2\right) e^{-\alpha|x|^2}.
\]
Thus no matter how we choose $\alpha$, \ef{eq:4.8} fails to hold near the origin.
\end{remark}

\begin{proof}[Proof of Theorem \ref{thm:1.1}]
By Proposition \ref{prop:3.6} and Proposition \ref{prop:4.2}, 
there exists $u_0 \in \HT \setminus \{ 0 \}$ such that
\begin{equation} \label{eq:4.9}
I(u_0)= \sigma =m \quad \text{and} \quad J(u_0)=0,
\end{equation}
namely, $u_0$ is a ground state solution of \ef{eq:1.1}.
We claim that 
$u_0$ can be assumed to be real-valued up to phase shift.

For this purpose, we argue as in \cite{AW} and show that
\begin{equation} \label{eq:4.10}
\big| \nabla |u_0| \big| = | \nabla u_0 | \quad \text {a.e. in} \ \mathbb{R}^3.
\end{equation}
Indeed suppose by contradiction that 
$\mathcal{L} \left( \left\{ x \in \R^3 \mid 
\big| \nabla |u_0(x)| \big| < | \nabla u_0(x) | \right\} \right)>0$,
where $\mathcal{L}(A)$ is the Lebesgue measure for the set $A \subset \R^3$.
Then it follows that
\begin{equation} \label{eq:4.11}
\sigma = I(u_0) > I(|u_0|), \quad 
0 = J(u_0) > J(|u_0|) \quad \text {and} \quad 
K(u_0) > K(|u_0|).
\end{equation}
Moreover by Lemma \ref{lem:3.9}, there exists a unique $\la_0>0$ such that
$v_0(x) := \la_0^2 |u_0|(\la_0 x) \in \CM$.
Then from \ef{eq:3.13}, \ef{eq:4.9} and \ef{eq:4.11}, we find that
\[
\begin{aligned}
\sigma &= I(u_0)-\frac{1}{3} J(u_0) = K(u_0) > K(|u_0|) \\
&= I(|u_0|)-\frac{1}{3} J(|u_0|) \\
&\ge I(v_0)-\frac{\la_0^3}{3} J(|u_0|) 
> I(v_0) \ge \sigma,
\end{aligned}
\]
which is a contradiction and hence \ef{eq:4.10} holds.

Now using the polar form, we can write 
$u_0(x)= |u_0(x)| e^{i \eta(x)}$.
Then a direct calculation shows that
\[
| \nabla u_0|^2= \big| \nabla |u_0| \big|^2+|u_0|^2|\nabla \eta|^2,
\]
yielding that $\nabla \eta \equiv 0$ a.e. in $\R^3$ from \ef{eq:4.10}.
This implies that $u_0(x)= e^{i \theta} |u_0(x)|$ for some $\theta \in \R$
and hence $u_0$ can be assumed to be real-valued 
up to phase shift.
\end{proof}

\section{Relation between action GSS and energy GSS}

In this section, we investigate the relation between
action ground state solutions and 
energy ground state solutions of \ef{eq:1.1}. 

For given $\mu>0$, let $u_{\mu} \in \HT$ be an energy ground state solution
of \ef{eq:1.1}, that is,
\[
\begin{aligned}
E(u_{\mu}) &= \inf_{B(\mu)} E(u), \quad 
B_\mu =\left\{ u \in \HT \mid \| u \|_2^2 = \mu \right\}, \\
E(u) &:= \frac{1}{2} A(u)-\frac{1}{p+1} C(u)+e^2 D(u) + 2e^2 E_1(u).
\end{aligned}
\]
By the result in \cite{CW4}, if 
\ef{Ass}-\ef{ASS} is satisfied, 
$2<p< \frac{7}{3}$, $\mu > 2 \cdot 2^{\frac{1}{2p-4}} \mu^*$ and 
$e^2 \left( \| \rho\|_{\frac{6}{5}}+ \| x \cdot \nabla \rho \|_{\frac{6}{5}} \right) 
\le \rho_0$ for sufficiently small $\rho_0>0$, 
$u_{\mu}$ exists and the corresponding Lagrange multiplier 
$\omega= \omega_{\mu}$ is positive.
Here $\mu^*$ is the threshold value for $c_{\infty}(\mu)$ defined in \ef{eq:1.8}.
Especially $u_{\mu}$ is a nontrivial solution of \ef{eq:1.1} with 
$\omega= \omega_{\mu}$.

To clarify the dependence with respect to $\omega$, we write
$I=I_{\omega}$, $J=J_{\omega}$, $\CS= \CS_{\omega}$ and $m=m_{\omega}$.
Under these preparations, we have the following result.

\begin{proposition} \label{prop:5.1}
Let $\mu > 2 \cdot 2^{\frac{1}{2p-4}} \mu^*$ be given 
and suppose that $2<p<\frac{7}{3}$.
Under the assumptions in Theorem \ref{thm:1.1} and \ef{Ass}, 
the energy ground state solution $u_{\mu}$ is 
an action ground state solution of \ef{eq:1.1} with $\omega=\omega_{\mu}$.
\end{proposition}

\begin{proof}
Since $u_{\mu}$ is a nontrivial critical point of $I_{\om}$, we have
\begin{equation} \label{eq:5.1}
u_{\mu} \in \CS_{\om} \quad \text{and} \quad 
m_{\om} \le I_{\om}(u_{\mu}).
\end{equation}
Thus it suffices to show that $I_{\om}(u_{\mu}) \le m_{\om}$.

Now let $w_{\mu} \in \HT \setminus \{ 0\}$ be an action ground state solution
of \ef{eq:1.1} with $\omega= \om$,
which exists by Theorem \ref{thm:1.1}.
Then by \ef{eq:3.13} and $J_{\om}(w_{\mu})=0$, it follows that
\begin{equation} \label{eq:5.2}
I_{\om}(w_\mu) - I_{\om}\left( (w_\mu)_{\la}\right) \ge
\frac{1-\la^3}{3} J_{\om} (w_\mu)=0 \quad \text{for all} \ \ \la>0,
\end{equation}
where $(w_{\mu})_\la= \la^2 w_{\mu}(\la x)$.
Taking $\la_{\mu} := \frac{\| u_{\mu} \|_2^2}{\| w_\mu \|_2^2}$, one has
\[
\left\| (w_\mu)_{\la_\mu}\right\|_2^2 = \| u_\mu \|_2^2
\quad \text{and hence} \quad
E(u_\mu) \le E\left( (w_\mu)_{\la_\mu} \right).
\]
Thus from \ef{eq:5.1} and \ef{eq:5.2}, we obtain
\[
\begin{aligned}
m_{\om} \le I_{\om}(u_\mu)
&= E(u_\mu)+ \frac{\om}{2} \| u_\mu\|_2^2 \\
&\le E\left( (w_\mu)_{\la_\mu}\right)
+\frac{\om}{2} \left\| (w_\mu)_{\la_\mu}\right\|_2^2 \\
&= I_{\om}\left( (w_\mu)_{\la_\mu}\right) 
\le I_{\om} (w_\mu)=m_{\om},
\end{aligned}
\]
from which we conclude.
\end{proof}

In order to consider the converse, let $\Omega(\mu)$ be the set of 
Lagrange multipliers associated with energy ground state solutions for $B_{\mu}$,
namely
\begin{align*}
\Omega(\mu) := \left\{ \right.
\omega_\mu>0 \mid  
&\ \hbox{$\omega_\mu$ is the Lagrange multiplier 
associated with an energy ground state } \\
& \hbox{of \ef{eq:1.1} under the constraint $B_{\mu}$}
\left. \right\}.
\end{align*}
By the result in \cite{CW4}, we know that $\Omega(\mu) \ne \emptyset$
for every $\mu> 2 \cdot 2^{\frac{1}{2p-4}} \mu^*$, when $2<p<\frac{7}{3}$ and
$e^2(\| \rho \|_{\frac{6}{5}}+\| x \cdot \nabla \rho \|_{\frac{6}{5}})$ is small. 
Moreover for any $\omega_\mu \in \Omega(\mu)$, 
there exists an action ground state solution $w_{\mu} \in \HT \setminus \{ 0 \}$ 
of \ef{eq:1.1} with $\omega= \om$ by Theorem \ref{thm:1.1}.

In this setting, the following result holds.

\begin{proposition} \label{prop:5.2}
Let $\mu> 2 \cdot 2^{\frac{1}{2p-4}} \mu^*$ be given 
and suppose that $2<p<\frac{7}{3}$.
Under the assumptions of Theorem \ref{thm:1.1} and \ef{Ass},
$w_{\mu}$ is an energy ground state solution of \ef{eq:1.1} 
under the constraint $B_{\mu}$.
\end{proposition}

\begin{proof}
Since $w_{\mu}$ is an action ground state solution of \ef{eq:1.1}
with $\omega = \om$, by using \ef{eq:3.13}, we have $J_{\om}(w_{\om})=0$ and
\begin{equation} \label{eq:5.3}
I_{\om} (w_\mu)-I_{\om} \left( (w_\mu)_\la \right)
\ge \frac{(1-\la)^2 \om}{6} \| w_\mu \|_2^2 \ge 0
\quad \text{for all} \ \ \la>0.
\end{equation}

Now let $\tilde{u}_{\mu}$ be an energy ground state solution of \ef{eq:1.1}
under the constraint $B_{\mu}$ whose Lagrange multiplier coincides with $\om$.
Then if follows that 
\[
m_{\om} \le I_{\om}(\tilde{u}_{\mu}) \quad \text{and} \quad
E(\tilde{u}_{\mu}) \le E(u) \quad \text{for any} \ u \in \HT
\ \text{with} \ \| u \|_2^2 = \mu.
\]
Especially choosing $\tilde{\la}_\mu =\frac{\mu}{\| w_\mu \|_2^2}$, we have
\begin{equation} \label{eq:5.4}
\left\| (w_\mu )_{\tilde{\la}_\mu} \right\|_2^2 = \mu = \| \tilde{u}_\mu \|_2^2
\quad \text{and} \quad
E(\tilde{u}_\mu) \le E\left( (w_\mu)_{\tilde{\la}_\mu}\right).
\end{equation}
Then from \ef{eq:5.3} and \ef{eq:5.4}, we deduce that
\[
\begin{aligned}
m_{\om} \le I_{\om}(\tilde{u}_\mu)
&= E(\tilde{u}_\mu) + \frac{\om}{2} \| \tilde{u}_\mu \|_2^2 \\
&\le E\left( (w_\mu)_{\tilde{\la}_\mu} \right)
+\frac{\om}{2} \left\| (w_\mu)_{\tilde{\la}_\mu} \right\|_2^2 \\
&= I_{\om}\left( (w_\mu)_{\tilde{\la}_{\mu}} \right) 
\le I_{\om}(w_\mu) = m_{\om},
\end{aligned}
\]
which yields that $I_{\om} \left( (w_\mu)_{\tilde{\la}_{\mu}} \right)
=I_{\om}(w_\mu)$.
Going back to \ef{eq:5.3}, one finds that
\[
0 \ge \frac{(1-\tilde{\la}_\mu)^2 \om}{6} \| w_\mu \|_2^2 \ge 0,
\]
from which we conclude that $\tilde{\la}_{\mu}=1$ and hence
\[
E(w_\mu) = E(\tilde{u}_\mu) = \inf_{B_\mu} E(u).
\]
This completes the proof.
\end{proof}

\begin{remark} \label{rem:5.3}
{\rm(i)} By Proposition \ref{prop:5.2}, if $\omega_\mu \in \Omega(\mu)$,
every action ground state solution of \ef{eq:1.1} with $\omega= \om$
share the same $L^2$-norm.

{\rm(ii)} Although $\Omega(\mu) \ne \emptyset$ for every $\mu>2 \cdot 2^{\frac{1}{2p-4}} \mu^*$,
we don't know whether $\Omega(\mu)$ is an interval.
Moreover we don't know for given $\omega_\mu \in \Omega(\mu)$,
there exists a unique $\mu>0$ such that 
the corresponding Lagrange multiplier coincides with $\omega_{\mu}$.
In other words, it is not known if
\begin{equation} \label{eq:5.5}
\mu \mapsto \om: (0,\infty) \to (0,\ell) \ 
\text{is one to one mapping for some} \ \ell>0.
\end{equation}

\ef{eq:5.5} is important in investigating further properties of ground state solutions,
such as uniqueness and stability.
However \ef{eq:5.5} is known only in limited situations;
see \cite{CW, CW4, DST}. 
In particular, it is not known whether \ef{eq:5.5} holds true 
even when $\rho \equiv 0$.
\end{remark}

\begin{proof}[Proof of Theorem \ref{thm:1.2}]
The claim follows by Propositions \ref{prop:5.1} and \ref{prop:5.2}.
\end{proof}

\section{Existence of a radial ground state solution for $1<p < 2$}

In this section, we consider the case $1<p<2$,
for which arguments in sections 3-5 do not work well.
Especially we cannot handle the energy inequality \ef{eq:3.13} in Lemma \ref{lem:3.7}
because the coefficient in front of $\| u \|_{p+1}^{p+1}$ becomes negative
when $1<p<2$.

In the case $1<p<2$, we only have the following weak result.

\begin{theorem} \label{thm:6.1}
Suppose that $1<p<2$. 
Assume $\rho \in L^{\frac{6}{5}}(\R^3)$ and $\rho(x)=\rho(|x|)$.
There exists $e_0>0$ such that if $0< e \le e_0$,
\ef{eq:1.1} has a radial ground state solution $u_0$, that is, $u_0$ satisfies
\[
I(u_0)= \inf_{u \in \CS_{rad}} I(u), \quad
\CS_{rad} = \{ u \in H^1_{rad}(\R^3) \mid I'(u)=0 \},
\]
where $H^1_{rad}(\R^3) = \{ u \in H^1(\R^3,\R) \mid u(x)=u(|x|) \}$.
\end{theorem}

To this aim, let us recall the functional $I_{\infty}$ for $\rho \equiv 0$,
which was defined by
\[
I_{\infty}(u) := \frac{1}{2} A(u) +\frac{\omega}{2} B(u) 
-\frac{1}{p+1} C(u)+ e^2 D(u) \quad \text {for} \ \ u \in \HT.
\]
When $1<p<2$, the following properties are known; 
see \cite[Theorem 4.3]{R}.

\begin{proposition} \label{prop:6.2}
Suppose that $1<p<2$ and $e>0$. Then it holds:
\begin{enumerate}
\item[\rm(i)] $\dis \inf_{H^1_{rad}(\mathbb{R}^3)} I_{\infty}(u)>-\infty$,

\item[\rm(ii)] $I_{\infty}$ satisfies the Palais-Smale condition on $H^1_{rad}(\R^3)$.
\end{enumerate}

\end{proposition}

Now by Lemma \ref{lem:2.2}, one finds that
\[
I(u) = I_{\infty}(u)+e^2 E_1(u) 
\ge I_{\infty}(u) - C e^2 \| \rho \|_{\frac{6}{5}} \| u \|^2
\quad \text{for any} \ u \in \HT,
\]
from which we conclude that $\dis \inf_{H^1_{rad}(\mathbb{R}^3)} I(u)>-\infty$.
Moreover since $E_1$ is a compact operator by Lemma \ref{lem:2.4}, 
$I$ also satisfies the Palais-Smale condition on $H^1_{rad}(\R^3)$.

\begin{proof}[Proof of Theorem \ref{thm:6.1}]
Let us define the energy functional $I_0$ for $e=0$:
\[
I_0(u) =\frac{1}{2} A(u) +\frac{\omega}{2} B(u) -\frac{1}{p+1} C(u) 
\quad \text {for} \ u \in \HT.
\]
It is standard to show that there exists $\tilde{u} \in H^1_{rad}(\R^3)$
such that $I_0(\tilde{u})<0$.
Then by Lemma \ref{lem:2.2}, one has
\[
\begin{aligned}
I(\tilde{u}) &= I_0(\tilde{u}) +e^2 D(\tilde{u}) +e^2 E_1(\tilde{u}) \\
&\le I_0(\tilde{u}) +C e^2 \| \tilde{u} \|^4 
+C e^2 \| \rho \|_{\frac{6}{5}} \| \tilde{u} \|^2.
\end{aligned}
\]
Choosing a small $e_0>0$, we have $I(\tilde{u})<0$ for $0<e \le e_0$
and hence $\dis \inf_{u \in H^1_{rad}(\R^3)} I(u)<0$.

By Proposition \ref{prop:6.2}, we are able to apply Ekeland's variational principle.
Then there exists $u_0 \in H^1_{rad}(\R^3) \setminus \{ 0 \}$ such that
\[
I(u_0)= \min_{u \in H^1_{rad}(\R^3)} I(u)<0.
\]
Particularly, $u_0$ is a nontrivial solution of \ef{eq:1.1}.
Moreover since $u_0$ is a global minimizer on $H^1_{rad}(\R^3)$, it follows that
\[
I(u_0) \ge \inf_{u \in \CS_{rad}} I(u) 
\ge \inf_{u \in H^1_{rad}(\R^3)} I(u) = I(u_0),
\]
yielding that $u_0$ is a radial ground state solution of \ef{eq:1.1}.
\end{proof}

\begin{remark} \label{rem:6.3}
{\rm(i)}  
In \cite{CW4}, we have shown that \ef{eq:1.1} has an energy ground state solution
if $1<p<2$ and $e$ is sufficiently small.
Unfortunately unlike the case $2<p<\frac{7}{3}$,
we cannot say anything about the relation
between the radial ground state solution obtained in Theorem \ref{thm:6.1}
and the energy ground state solution in \cite{CW4}.

{\rm(ii)} In the case $p=2$, it remains open whether \ef{eq:1.1} has 
a nontrivial solution for any $\omega>0$ and small $e>0$.
Moreover in \cite{CW}, thanks to the scaling invariance of \ef{eq:1.3} when $p=2$,
it was proved that an energy ground state solution of \ef{eq:1.3}
is an action ground state solution and vise versa.
However due to the loss of scaling invariance for $\rho \not\equiv 0$,
we don't know if an energy ground state solution of \ef{eq:1.1} 
obtained in \cite{CW4} is an action ground state of \ef{eq:1.1}.
\end{remark}

\section{The case $\rho$ is a characteristic function}

In this section, we consider the case where the doping profile $\rho$
is a characteristic function, 
which appears frequently in physical literatures \cite{Je, MRS, Sel}.
More precisely, let $\{ \Omega_i \}_{i=1}^m \subset \R^3$ be
disjoint bounded open sets with smooth boundary.
For $\alpha_i>0$ $(i=1,\cdots, m)$, 
we assume that the doping profile $\rho$ has the form:
\begin{equation} \label{eq:7.1}
\rho(x)= \sum_{i=1}^m \alpha_i \chi_{\Omega_i}(x), \quad
\chi_{\Omega_i}(x)=
\begin{cases} 
1 & (x \in \Omega_i), \\
0 & (x \notin \Omega_i).
\end{cases}
\end{equation}
In this case, $\rho$ cannot be weakly differentiable so that the assumption
\ef{eq:1.5} does not make sense.
Even so, we are able to obtain the existence of ground state solutions 
by imposing some smallness condition related with $\alpha_i$ and $\Omega_i$.

To state our main result for this case, 
let us put $\dis L:= \sup_{x \in \partial \Omega} |x| < \infty$.
A key is the following \textit{sharp boundary trace inequality} 
due to \cite[Theorem 6.1]{A},
which we present here according to the form used in this paper.

\begin{proposition} \label{prop:7.1}
Let $\Omega \subset \R^3$ be a bounded domain with smooth boundary
and $\gamma: H^1(\Omega) \to L^2(\partial \Omega)$ be 
the trace operator.
Then it holds that
\[
\int_{\partial \Omega} | \gamma(u) |^2 \,dS
\le \kappa_1(\Omega) \int_{\Omega} |u|^2 \,dx
+ \kappa_2(\Omega) 
\left( \int_{\Omega} |u|^2 \,dx \right)^{\frac{1}{2}}
\left( \int_{\Omega} | \nabla u|^2 \,dx \right)^{\frac{1}{2}}
\quad \hbox{for any} \ u \in H^1(\Omega),
\]
where $\kappa_1(\Omega)= \frac{|\partial \Omega|}{|\Omega|}$,
$\kappa_2(\Omega) = \big\| | \nabla w| \big\|_{L^{\infty}(\partial\Omega)}$
and $w$ is a unique solution of the torsion problem:
\[
\Delta w = \kappa_1(\Omega) \ \hbox{in} \ \Omega, 
\quad \frac{\partial w}{\partial n} = 1 \ \hbox{on} \ \partial \Omega.
\]
\end{proposition}

In relation to the size of $\rho$, we define
\[
D(\Omega) := L | \Omega |^{\frac{1}{6}} | 
\left( L \| H \|_{L^2(\partial \Omega)} + |\partial \Omega |^{\frac{1}{2}} \right)
\left( \kappa_1(\Omega) | \Omega |^{\frac{1}{3}} + \kappa_2 (\Omega) 
\right)^{\frac{1}{2}},
\]
where $H$ is the mean curvature of $\partial \Omega$.

\begin{remark} \label{rem:7.2}
It is known that $\kappa_2(\Omega) \ge 1$; see \cite{A}.
Then by the isoperimetric inequality in $\R^3$:
\[
| \partial \Omega | \ge 3 | \Omega |^{\frac{2}{3}} | B_1 |^{\frac{1}{3}},
\]
and the fact $| \Omega | \le |B_L(0)| = L^3 |B_1|$, we find that
\begin{equation} \label{eq:7.2}
D(\Omega) \ge
\left( \frac{|\Omega|}{|B_1|} \right)^{\frac{1}{3}} |\Omega|^{\frac{1}{6}}
\cdot \sqrt{3} |\Omega|^{\frac{1}{3}} |B_1|^{\frac{1}{6}}
\left( 3 |B_1|^{\frac{1}{3}} +1 \right)^{\frac{1}{2}} 
=C |\Omega|^{\frac{5}{6}} = C \| \chi_{\Omega} \|_{L^{\frac{6}{5}}(\R^3)},
\end{equation}
where $C$ is a positive constant independent of $\Omega$. 
\end{remark}

Under these preparations, we have the following result.

\begin{theorem} \label{thm:7.3}
Suppose that $2<p<5$ and assume \ef{eq:7.1}.
There exists $\rho_0>0$ such that if
\[
e^2 \sum_{i=1}^m \alpha_i D(\Omega_i) \le \rho_0,
\]
then \ef{eq:1.1} has a ground state solution $u_0$.
Moreover the statement of Theorem \ref{thm:1.2} holds true.
\end{theorem}

Note that when $1<p<2$, we have only assumed that
$\rho(x) \in L^{\frac{6}{5}}(\R^3)$ and $\rho(x)=\rho(|x|)$ in Theorem \ref{thm:6.1},
which covers the case \ef{eq:7.1}.

\smallskip
We mention that the first part $x \cdot \nabla \rho(x)$ 
and $x \cdot (D^2 \rho(x) x)$ appeared
was the definition of $E_2(u)$ and $E_3(u)$ in \ef{eq:2.5}.
Under the assumption \ef{eq:7.1}, we replace them by
\begin{align*}
E_1(u) &= - \frac{1}{4} \sum_{i=1}^m \alpha_i 
\int_{\Omega_i} S_0(u) \,dx, \\
E_2(u) &:= - \frac{1}{2} \sum_{i=1}^m \alpha_i 
\int_{\partial \Omega_i} S_0(u) x \cdot n_i \,dS_i, \\
E_3(u) &:= - \frac{1}{2} \sum_{i=1}^m \alpha_i 
\int_{\partial \Omega_i} H_i(x) S_0(u) (x \cdot n_i)^2 \,dS_i, 
\end{align*}
where $n_i$ is the unit outward normal on $\partial \Omega_i$.
Indeed we have the following.

\begin{lemma} \label{lem:7.4}
It holds that
\begin{align*}
\lim_{R \to \infty} \int_{B_R(0)} S_0(u) u x \cdot \nabla \bar{u} \,dx 
&= -10 E_1(u) + E_2(u), \\
\lim_{R \to \infty} \int_{B_R(0)} S_1(u) u x \cdot \nabla \bar{u} \,dx 
&= - 6 E_2(u) - E_3(u).
\end{align*}

\end{lemma}

\begin{proof}
For simplicity, let us consider the case $m=1$ and $\alpha=1$.
First by the divergence theorem and the fact $S_0(u)|u|^2 \in L^1(\R^3)$, 
one finds that
\[
\begin{aligned}
&\lim_{R \to \infty} \int_{B_R(0)} S_0(u) u x \cdot \nabla \bar{u} \,dx 
= -\frac{1}{8 \pi} \lim_{R \to \infty} \int_{\Omega} \int_{B_R(0)} 
\frac{u(y) y \cdot \nabla \overline{u(y)}}{|x-y|} \,d y \,d x \\
&= \frac{1}{16 \pi} \int_{\Omega} \int_{\mathbb{R}^3} 
\frac{|u(y)|^2 \operatorname{div}_y y}{|x-y|} \,d y \,d x
+\frac{1}{16 \pi} \int_{\Omega} \int_{\mathbb{R}^3} 
|u(y)|^2 y \cdot \nabla_y\left(\frac{1}{|x-y|}\right) \,d y \,dx \\
&= \frac{3}{16 \pi} \int_{\Omega} \int_{\mathbb{R}^3} 
\frac{|u(y)|^2}{|x-y|} \,d y \,d x
+\frac{1}{16 \pi} \int_{\Omega} \int_{\mathbb{R}^3}
|u(y)|^2 \frac{y \cdot(x-y)}{|x-y|^3} \,d y \,d x. 
\end{aligned}
\]
Using the identity $y \cdot (x-y) = -|x-y|^2+x \cdot (x-y)$,
the Fubini theorem and the divergence theorem, we get
\[
\begin{aligned}
&\lim_{R \to \infty} \int_{B_R(0)} S_0(u) u x \cdot \nabla \bar{u} \,dx \\
%E_1^{\prime}(u) x \cdot \nabla u
&= \frac{1}{8 \pi} \int_{\Omega} \int_{\mathbb{R}^3} 
\frac{|u(y)|^2}{|x-y|} \,d y \,d x
-\frac{1}{16 \pi} \int_{\mathbb{R}^3} \int_{\Omega}
|u(y)|^2 x \cdot \nabla_x \left(\frac{1}{|x-y|}\right) \,d x \,d y \\
&= \frac{1}{8 \pi} \int_{\Omega} \int_{\R^3} 
\frac{|u(y)|^2}{|x-y|} \,dy \,dx 
-\frac{1}{16 \pi} \int_{\mathbb{R}^3} \int_{\Omega} 
\operatorname{div}_x \left( \frac{|u(y)|^2 x}{|x-y|}\right) \,d x \,d y 
+\frac{3}{16 \pi} \int_{\R^3} \int_{\Omega} \frac{|u(y)|^2}{|x-y|} \,d x \,d y \\
&= \frac{5}{16 \pi} \int_{\Omega} \int_{\mathbb{R}^3} 
\frac{|u(y)|^2}{|x-y|} \,d y \,d x
-\frac{1}{16 \pi} \int_{\mathbb{R}^3} \int_{\Omega} 
\frac{|u(y)|^2}{|x-y|} x \cdot n \,d S \,d y \\
&= \frac{5}{2} \int_{\Omega} S_0(u) \,d x
-\frac{1}{2} \int_{\partial \Omega} S_0(u) x \cdot n \,d S
=-10 E_1(u)+ E_2(u).
\end{aligned}
\]
Similarly, we have
\begin{align} \label{eq:7.3}
&\lim_{R \to \infty} \int_{B_R(0)} S_1(u) u x \cdot \nabla \bar{u} \,dx \notag \\
%E_2^{\prime}(u) x \cdot \nabla u
&= -\frac{1}{8 \pi} \lim_{R \to \infty} \int_{\partial \Omega} \int_{B_R(0)} 
\frac{u(y) y \cdot \nabla \overline{u(y)}}{|x-y|} x \cdot n \,dy \,dS \notag \\
&= \frac{1}{8 \pi} \int_{\partial \Omega} \int_{\mathbb{R}^3} 
\frac{|u(y)|^2 x \cdot n}{|x-y|} \,dy \,dS
-\frac{1}{16 \pi} \intR \int_{\partial \Omega} 
|u(y)|^2 x \cdot \nabla_x \left( \frac{1}{|x-y|} \right) x \cdot n \,dS \,dy \notag \\
&= \frac{1}{8 \pi} \int_{\partial \Omega} \int_{\mathbb{R}^3} 
\frac{|u(y)|^2 x \cdot n}{|x-y|} \,dy \,dS
-\frac{1}{16 \pi} \intR \int_{\partial \Omega}
\operatorname{div}_x 
\left( \frac{|u(y)|^2 (x \cdot n) x}{|x-y|} \right) \,dS \,dy \notag \\
&\quad +\frac{3}{16 \pi} \intR \int_{\partial \Omega} 
\frac{|u(y)|^2 x \cdot n}{|x-y|} \,dS \,dy
+\frac{1}{16 \pi} \intR \int_{\partial \Omega} 
\frac{|u(y)|^2 x \cdot \nabla_x (x \cdot n)}{|x-y|} \,dS \,dy \notag \\
&= \frac{5}{2} \int_{\partial \Omega} S_0(u) x \cdot n \,d S
-\frac{1}{2} \int_{\partial \Omega} 
\operatorname{div}_x 
\big( S_0(u) (x \cdot n) x \big) \,dS
+\frac{1}{2} \int_{\partial \Omega} 
S_0(u) x \cdot \nabla_x (x \cdot n) \,dS. 
\end{align}
Applying the surface divergence theorem (see e.g. \cite[7.6]{Si})
and noticing that $\partial (\partial \Omega)= \emptyset$, it follows that
\[
\begin{aligned}
\int_{\partial \Omega} \operatorname{div}_x (S_0(u) (x \cdot n) x) \,d S 
& =-\int_{\partial \Omega} \big( S_0(u)(x \cdot n) x \big)^{\perp} \cdot \vec{H} \,d S 
 =-\int_{\partial \Omega} H(x)S_0(u) (x \cdot n)^2 \,d S,
\end{aligned}
\]
where $x^{\perp}$ is the normal component of $x$ and 
$\vec{H}$ is the mean curvature vector $\vec{H}= H n$.
Finally since $x \cdot \nabla_x (x \cdot n)= x \cdot n$, 
we deduce from \ef{eq:7.3} that
\[
\begin{aligned}
&\lim_{R \to \infty} \int_{B_R(0)} S_1(u) u x \cdot \nabla \bar{u} \,dx \\
&= 3 \int_{\partial \Omega} S_0(u) x \cdot n \,d S
+\frac{1}{2} \int_{\partial \Omega} H(x) S_0(u) (x \cdot n)^2 \,d S =-6 E_2(u)-E_3(u),
\end{aligned}
\]
which ends the proof.
The general case can be shown by summing up the integrals.
\end{proof}

By Lemma \ref{lem:7.4}, the Pohozev identity can be reformulated as follows.

\begin{lemma} \label{lem:7.5}
Under the assumption \ef{eq:7.1}, 
the functionals $P(u)$ defined in \ef{eq:2.10} and $Q(u)$ in \ef{eq:3.7} 
have the same form.
\end{lemma}

\begin{proof}
As we have mentioned in the proof of Lemma \ref{lem:2.5}, 
the Pohozaev identity can be shown by multiplying $x \cdot \nabla \bar{u}$
and $ex \cdot S_0(u)$ by \ef{eq:1.1}, 
integrating them over $B_R(0)$ and passing to a limit $R \to \infty$.
Then we are able to obtain \ef{eq:2.10} and \ef{eq:3.7} by Lemma \ref{lem:7.4}.
\end{proof}

Next we establish estimates for $E_1$, $E_2$ and $E_3$.

\begin{lemma} \label{lem:7.6}
For any $u \in H^1(\R^3,\C)$, $E_1$, $E_2$ and $E_3$ satisfy the estimates:
\[
\begin{aligned}
|E_1(u)| &\le C \sum_{i=1}^m \alpha_i |\Omega_i|^{\frac{5}{6}}
\| \nabla S_0(u) \|_{2}, \\
|E_2(u)| &\le C \sum_{i=1}^m \alpha_i D(\Omega_i)
\| \nabla S_0(u) \|_{2}, \\
|E_3(u)| &\le C \sum_{i=1}^m \alpha_i D(\Omega_i)
\| \nabla S_0(u) \|_{2},
\end{aligned}
\]
where $C>0$ is a constant independent of $\Omega_i$.
\end{lemma}

\begin{proof}
First we observe that 
\[
|E_1(u)| \le \frac{1}{4} \sum_{i=1}^m \alpha_i \int_{\Omega_i} |S_0(u)|\,dx
\le \frac{1}{4} \sum_{i=1}^m \alpha_i 
\left( \int_{\Omega_i} |S_0(u)|^6 \,dx \right)^{\frac{1}{6}}
\left( \int_{\Omega_i} \,dx \right)^{\frac{5}{6}},
\]
from which the estimate for $E_1$ can be obtained by the Sobolev inequality. 
Next by Proposition \ref{prop:7.1}, 
the H\"older inequality and the Sobolev inequality, one has
\begin{align*}
|E_2(u)| & \le \frac{1}{2} \sum_{i=1}^m \alpha_i
\int_{\partial \Omega_i} |S_0(u)| |x| \,dS_i \\
&\le \frac{1}{2} \sum_{i=1}^m \alpha_i 
\left( \int_{\partial \Omega_i} |S_0(u)|^2 \,dS_i \right)^{\frac{1}{2}}
\left( \int_{\partial \Omega_i} |x|^2 \,dS_i \right)^{\frac{1}{2}} \\
&\le \frac{1}{2} \sum_{i=1}^m \alpha_i L_i | \partial \Omega_i|^{\frac{1}{2}}
\left( \kappa_1(\Omega_i) \| S_0(u) \|_{L^2(\Omega_i)}^2
+\kappa_2(\Omega_i) \| S_0(u) \|_{L^2(\Omega_i)}
\| \nabla S_0(u) \|_{L^2(\Omega_i)} \right)^{\frac{1}{2}} \\
\hspace{-1em} &\le \frac{1}{2} \sum_{i=1}^m \alpha_i L_i |\partial \Omega_i|^{\frac{1}{2}}
\left( \kappa_1(\Omega_i) |\Omega_i|^{\frac{2}{3}} \| S_0(u) \|_{L^6(\R^3)}^2
+\kappa_2(\Omega_i) | \Omega_i |^{\frac{1}{3}} \| S_0(u) \|_{L^6(\R^3)}
\| \nabla S_0(u) \|_{L^2(\R^3)} \right)^{\frac{1}{2}} \\
&\le C \sum_{i=1}^m \alpha_i L_i |\Omega_i|^{\frac{1}{6}}
|\partial \Omega_i|^{\frac{1}{2}} 
\left( \kappa_1(\Omega_i)|\Omega_i|^{\frac{1}{3}} 
+\kappa_2(\Omega_i) \right)^{\frac{1}{2}} 
\| \nabla S_0(u) \|_{L^2(\R^3)} \\
&\le C \sum_{i=1}^m \alpha_i D(\Omega_i) 
\| \nabla S_0(u) \|_{L^2(\R^3)}.
\end{align*}
Similarly, we obtain
\[
\begin{aligned}
|E_3(u)| &\le \frac{1}{2} \sum_{i=1}^m \alpha_i \int_{\partial \Omega_i}
|H_i| |S_0(u)| |x|^2 \,d S_i \\
&\le \frac{1}{2} \sum_{i=1}^m \alpha_i L_i^2 \|H_i\|_{L^2(\partial \Omega_i)}
\|S_0(u)\|_{L^2(\partial \Omega_i)} \\
&\le C \sum_{i=1}^m \alpha_i L_i^2 \| H_i \|_{L^2(\partial \Omega_i)}
| \Omega_i |^{\frac{1}{6}} \left( \kappa_1(\Omega_i)|\Omega_i|^{\frac{1}{3}} 
+\kappa_2(\Omega_i) \right)^{\frac{1}{2}} 
\| \nabla S_0(u) \|_{L^2(\R^3)} \\
&\le C \sum_{i=1} \alpha_i D(\Omega_i) \| \nabla S_0(u)\|_{L^2(\R^3)}.
\end{aligned}
\]
This completes the proof.
\end{proof}

Our next step is to modify the proof of the energy inequality in Lemma \ref{lem:3.7}.
For this purpose, we prove the following.

\begin{lemma} \label{lem:7.7}
Let $\Omega \subset \R^3$ be a bounded domain with smooth boundary
and put
\[
\Omega(\la) := \int_{\la \Omega} S_0(u)(x) \,dx
= \la^3 \int_{\Omega} S_0(u)(\la y) \,dy.
\]
Then it holds that
\[
\begin{aligned}
\Omega'(\la) &= \la^2 \int_{\partial \Omega} S_0(u)(\la y) (y \cdot n) \,dS, \\
\Omega''(\la) &= -2\la \int_{\partial \Omega} S_0(u)(\la y) (y \cdot n) \,dS
- \la \int_{\partial \Omega} H(y) S_0(u)(\la y) (y \cdot n)^2 \,dS.
\end{aligned}
\]

\end{lemma}

\begin{proof}
First we observe that 
\[
\begin{aligned}
\Omega^{\prime}(\la)
&=3 \la^2 \int_{\Omega} S_0(u)(\la y) \,d y
+\la^2 \int_{\Omega} \nabla_y S_0(u)(\la y) \cdot y \,d y \\
&= 3 \la^2 \int_{\Omega} S_0(u)(\la y) \,dy
+\la^2 \int_{\Omega} \operatorname{div}_y \big( S_0(u)(\la y) y \big) \,dy
-\la^2 \int_{\Omega} S_0(u)(\la y) \operatorname{div}_y y \,dy \\
&= \la^2 \int_{\partial \Omega} S_0(u)(\la y) (y \cdot n) \,dS.
\end{aligned}
\]
Similarly by the surface divergence theorem, one has
\[
\begin{aligned}
\Omega^{\prime \prime}(\la) 
&= 2 \la \int_{\partial \Omega} S_0(u)(\la y) (y \cdot n) \,d S 
+\la \int_{\partial \Omega} \nabla_y S_0(u)(\la y) \cdot y (y \cdot n) \,d S \\
&= 2\la \int_{\partial \Omega} S_0(u)(\la y) (y \cdot n) \,dS \\
&\quad +\la \int_{\partial \Omega} 
\operatorname{div}_y \big( S_0(u)(\la y) (y \cdot n)y \big) \,dS 
-\la \int_{\partial \Omega} 
S_0(u)(\la y) \operatorname{div}_y \big( (y \cdot n) y \big) \,dS \\
&= 2 \la \int_{\partial \Omega} S_0(u)(\la y) (y \cdot n) \,dS
-\la \int_{\partial \Omega} 
H(y) S_0(u)(\la y) (y \cdot n)^2 \,dS \\
&\quad -\la \int_{\partial \Omega} 
S_0(u)(\la y) \operatorname{div}_y y (y \cdot n) \,dS 
-\la \int_{\partial \Omega} 
S_0(u)(\la y) y \cdot \nabla_y (y \cdot n) \,dS \\
&= -2 \la \int_{\partial \Omega} S_0(u)(\la y) (y \cdot n) \,d S
-\la \int_{\partial \Omega} H(y) S_0(u)(\la y) (y \cdot n)^2 \,d S.
\end{aligned}
\]
This completes the proof.
\end{proof}

\begin{remark} \label{rem:7.8}
Lemma \ref{lem:7.7} is related with the
''calculus of moving surfaces'' due to Hadamard;
see \cite[(38)-(39)]{Gri}.
\end{remark}

Using Lemma \ref{lem:7.6} and Lemma \ref{lem:7.7}, 
we can establish the energy identity as follows.

\begin{lemma} \label{lem:7.9}
Suppose that $2<p<5$ and let $u_\la(x)= \la^2 u(\la x)$ for $\la>0$.
Under the assumption \ef{eq:7.1}, 
there exist $\alpha>0$, $\beta>0$, $\rho_0>0$ 
independent of $e$, $\rho$, $\la$ such that
if $\| u\|_{p+1} \ge \delta$ for some $\delta>0$ independent of $e$, $\rho$
and $\dis e^2 \sum_{i=1}^m \alpha_i D(\Omega_i) \le \rho_0$,
then the following estimate holds.
\[
I(u)-I(u_\la) -\frac{1-\la^3}{3} J(u) 
\ge \frac{(1-\la)^2 \omega}{6} \| u\|_2^2
+ \frac{\alpha \delta^{p-1}}{12(p+1)} (1-\la)^2 \| u \|_{p+1}^2 
\quad \text{for all} \ \la>0.
\]
\end{lemma}

\begin{proof}
Under the notation of Lemma \ref{lem:7.7}, 
we can write the remainder term $R(\la,u)$ as follows:
\[
\begin{aligned}
R(\la,u) &= \frac{8-2\la^3}{3} e^2 E_1(u) -\frac{1-\la^3}{3} e^2 E_2(u)
+\frac{e^2 \la^{-1}}{2} \int_{\R^3} S_0(u) \rho (\la^{-1} x) \,dx \\
&= e^2 \sum_{i=1}^m \alpha_i \left\{
-\frac{4-\la^3}{6} \int_{\Omega_i} S_0(u) \,dx 
+\frac{1-\la^3}{6} \int_{\partial \Omega_i} S_0(u) x \cdot n \,dS_i
+\frac{\la^{-1}}{2} \int_{\la \Omega_i} S_0(u) \,dx \right\} \\
&= e^2 \sum_{i=1}^m \alpha_i \left\{ 
\frac{\la^3-1}{6} \left( \Omega_i(1) - \Omega_i^{\prime}(1) \right)
- \frac{\Omega_i(1)}{2} + \frac{\Omega_i(\la)}{2\la} \right\}.
\end{aligned}
\]
Let $T \ge 4$ be chosen so that $T^{2p-4} \ge 3$ and put
\[
G(\la) := \frac{\la^3-1}{6} \left( \Omega_i(1)-\Omega_i^{\prime}(1)\right)
-\frac{\Omega_i(1)}{2}+\frac{\Omega_i(\la)}{2 \la}.
\]
For $\la \ge T$, it follows that
\[
G(\la) \ge 
-\frac{\la^3-1}{6} \left( |\Omega_i(1)|+ |\Omega_i^{\prime}(1)|\right)
-\frac{1}{2} |\Omega_i(1)| 
\ge - \la^3 \left( \frac{2}{3} |\Omega_i(1)|
+\frac{1}{6} | \Omega_i^{\prime}(1)| \right).
\]
Similarly one has
\[
G(\la) \ge -\left(\frac{2}{3} |\Omega_i(1)|+\frac{1}{6}|\Omega_i^{\prime}(1)|\right) 
\ge -4(1-\la)^2 \left( \frac{2}{3} |\Omega_i(1)|
+\frac{1}{6} |\Omega_i^{\prime}(1)| \right)
\quad \text{for} \ 0 \le \la \le \frac{1}{2}.
\]
When $\frac{1}{2} \le \la \le T$, we see that $G(1)=G'(1)=0$ and
\[
\begin{aligned}
G^{\prime \prime}(\la) 
&= \la\left(\Omega_i(1)-\Omega_i^{\prime}(1)\right)
+\frac{1}{\la^3} \left( \Omega_i^{\prime}(\la)-\la \Omega_i^{\prime}(\la)
+\frac{\la^2}{2} \Omega_i^{\prime \prime}(\la)\right) \\
&\ge -\la \left( |\Omega_i(1)|+ |\Omega_i^{\prime}(1)| \right) 
-\frac{1}{\la^3} \left( |\Omega_i(\la)|+\la | \Omega_i^{\prime}(\la)|
+\frac{\la^2}{2} | \Omega_i^{\prime \prime}(\la)| \right) 
=: -\tilde{N}(\la) .
\end{aligned}
\]
Then by the Taylor theorem, 
there exists $\xi = \xi(\la) \in \left( \frac{1}{2}, T \right)$ such that
\[
G(\la) \ge -\frac{1}{2} \tilde{N}(\xi)(1-\la)^2 \quad \text{for} \ 
\frac{1}{2} \le \la \le T.
\]

Now by Proposition \ref{prop:7.1} and Lemma \ref{lem:7.7},
arguing similarly as Lemma \ref{lem:7.6}, one finds that
\[
\begin{aligned}
\frac{|\Omega_i(\la)|}{\la^3}
&\le \int_{\Omega_i} | S_0(u)(\la y)| \,d y 
\le |\Omega_i|^{\frac{5}{6}} \| S_0(u) (\la y) \|_{L^6(\R^3)} 
\le C \la^{-\frac{1}{2}} D(\Omega_i) \| \nabla S_0(u) \|_{L^2(\R^3)}, \\
\frac{|\Omega_i^{\prime}(\la)|}{\la^2}
&\le \int_{\partial \Omega_i} | S_0(u)(\la y)| |y| \,dS 
\le C \la^{-\frac{1}{2}} D(\Omega_i) \| \nabla S_0(u) \|_{L^2(\R^3)}, \\
\frac{|\Omega_i^{\prime \prime}(\la)|}{\la}
&\le 2 \int_{\partial \Omega_i} | S_0(u)(\la y)| |y| \,dS
+ \int_{\partial \Omega_i} |H_i(y)| | S_0(u)(\la y)| |y|^2 \,dS 
\le C \la^{-\frac{1}{2}} D(\Omega_i) \| \nabla S_0(u) \|_{L^2(\R^3)}.
\end{aligned}
\]
Then by Lemma \ref{lem:2.2}, we have
\[
\tilde{N}(\xi) \le T \left( |\Omega_i(1)|+|\Omega_i^{\prime}(1)|\right)
+ \sqrt{2} \left( C D(\Omega_i) \| \nabla S_0(u) \|_{L^2(\R^3)} \right) 
\le C D(\Omega_i ) \left( \| u\|_2^2+\| u\|_{p+1}^2\right)
\]
and hence
\[
R(\la, u) \ge \begin{cases}
\dis - \beta(1-\la)^2 e^2 \sum_{i=1}^m \alpha_i D(\Omega_i)
\left(\| u\|_2^2+\| u\|_{p+1}^2\right) &\text{for} \ 0 \le \la \le T, \medskip \\
\dis -\beta \la^3 e^2 \sum_{i=1}^m \alpha_i D(\Omega_i)
\left( \| u\|_2^2+\| u \|_{p+1}^2 \right) &\text{for} \ \la \ge T,
\end{cases}
\]
for some $\beta>0$ independent of $e$, $\rho$, $\la$.
The remaining parts can be shown 
in the same way as Lemma \ref{lem:3.7}.
\end{proof}

\begin{proof}[Proof of Theorem \ref{thm:7.3}]
By Lemma \ref{lem:7.5}, Lemma \ref{lem:7.6} and Lemma \ref{lem:7.9}, 
we are able to modify the proofs in Sections 3-5. 
\end{proof}

\bigskip
%%%%%%%%%%%%%%%%%%%%%%%%%%%%%%%%%%%%%%%%%%%%%%%%%%%%%%%%%%%%%%%%
\noindent {\bf Acknowledgements.}

The second author has been supported by JSPS KAKENHI Grant Numbers 
JP21K03317, JP24K06804.
This paper was carried out while the second author was staying 
at the University of Bordeaux. 
The second author is very grateful to all the staff of the University of Bordeaux 
for their kind hospitality.

%%%%%%%%%%%%%%%%%%%%%%%%%%%%%%%%%%%%%%%%%%%%%%%%%%%%%%%%%%%%%%%%

\end{document}